\documentclass[oneside, 12pt]{amsart}
\usepackage{amscd,amssymb,enumerate, amsmath,mathrsfs}
\usepackage{url}
\usepackage{tikz}
\usetikzlibrary{decorations.pathreplacing}
\usepackage[backref=page]{hyperref}
\usepackage{hyperref}
\usepackage[backref=page]{hyperref}

\newtheorem{theorem}{Theorem}[section]

\newtheorem{lemma}[theorem]{Lemma}
\newtheorem{proposition}[theorem]{Proposition}
\newtheorem{corollary}[theorem]{Corollary}

\theoremstyle{definition}

\newtheorem{example}[theorem]{Example}

\newtheorem{remark}[theorem]{Remark}
\newtheorem{emp}[theorem]{}
\newtheorem*{acknowledgement}{Acknowledgement}

\theoremstyle{remark}

\DeclareFontFamily{U}{wncy}{}
\DeclareFontShape{U}{wncy}{m}{n}{<->wncyr10}{}
\DeclareSymbolFont{mcy}{U}{wncy}{m}{n}
\DeclareMathSymbol{\Sh}{\mathord}{mcy}{"58}

\newcommand  {\Ker }  	{\operatorname{Ker}}

\newcommand  {\rank}    {\operatorname{rank}}

\def\mydate{\number\day\space\ifcase\month \or January\or February\or March\or 
April\or May\or June\or July\or
August\or September\or October\or November\or December\fi \space\number\year}

\DeclareFontFamily{U}{wncy}{}
\DeclareFontShape{U}{wncy}{m}{n}{<->wncyr10}{}
\DeclareSymbolFont{mcy}{U}{wncy}{m}{n}
\DeclareMathSymbol{\Sh}{\mathord}{mcy}{"58}
\numberwithin{equation}{section}
\if false
\makeatletter
    \let\c@equation\c@thm
    \let\c@figure\c@thm
   \let\c@table\c@thm
\makeatother
\fi

\begin{document}

\title{The critical polynomial of a graph}

\author[Dino Lorenzini]{Dino Lorenzini}
\address{Department of Mathematics, University of Georgia, Athens, GA 30602, USA}
\curraddr{}
\email{lorenzin@uga.edu}

\if false
05  Combinatorics {For finite fields, see 11Txx}
05C View Publications (1973-now) Graph theory {For applications of graphs, see 68R10, 81Q30, 81T15, 82B20, 82C20, 90C35, 92E10, 94C15}
05C50 View Publications (1980-now) Graphs and linear algebra (matrices, eigenvalues, etc.)
\fi


\begin{abstract}
Let $G$ be a connected graph on $n$ vertices with adjacency matrix $A_G$. Associated to $G$ is a polynomial $d_G(x_1,\dots, x_n)$ of degree $n$ in $n$ variables, obtained as the determinant of the matrix $M_G(x_1,\dots,x_n)$, where 
$M_G={\rm Diag}(x_1,\dots,x_n)-A_G$. We investigate in this article the set $V_{d_G}(r)$  of non-negative values taken by this polynomial when $x_1,
\dots, x_n \geq r \geq 1$. We show that $V_{d_G}(1) = {\mathbb Z}_{\geq 0}$.
We show that for a large class of graphs one also has $V_{d_G}(2) = {\mathbb Z}_{\geq 0}$. When $V_{d_G}(2) \neq  {\mathbb Z}_{\geq 0}$, we show that for many graphs $V_{d_G}(2) $ is dense in $ {\mathbb Z}_{\geq 0}$. We give numerical evidence that in many cases, the complement of $V_{d_G}(2) $ in $ {\mathbb Z}_{\geq 0}$ might in fact be finite. As a byproduct of our results, we show that every graph can be endowed with an arithmetical structure whose associated group is trivial.
\end{abstract}

\dedicatory{ \mydate}
\maketitle

\begin{section}{Introduction}

Given any integer $r$, we let as usual ${\mathbb Z}_{\geq r}$
denote the set of  integers greater than or equal to $r$.
For any polynomial $f \in {\mathbb Z}[x_1,\dots, x_n]$ and integer $r \geq 1$, consider the following set of
non-negative values 
$$V_f(r):= \{ f(a_1,\dots ,a_n) \mid a_1,\dots ,a_n \in {\mathbb Z}_{\geq r} \} \cap  {\mathbb Z}_{\geq 0} .$$
 Given a   polynomial $f$, we do not know of any general results that quantify the difference between
$V_f(1) $ and $  V_f(2)$, and more generally, provide insights on the decreasing chain of sets $V_f(1) \supseteq V_f(2) \supseteq V_f(3)\supseteq \dots$. It is clear that $V_f(1) \supset V_f(2)$ can be strictly decreasing. As the example below of 
the path $A_2$ on two vertices shows, the complement of $V_f(2) $ in $V_f(1)$ can be infinite. In this article, inspired by questions in algebraic geometry, we describe the sets $V_f(1) \supseteq V_f(2)$ for a large class of polynomials $f$ associated to graphs, and give evidence that in many cases, the complement of $V_f(2) $ in $ V_f(1)$ is finite for such polynomials.

Let $G$ denote a connected undirected graph on $n$ vertices $v_1,\dots, v_n$, with no self-loops.  
Let $A_G$ denote the associated symmetric adjacency matrix. Let 
$M_G$ denote the matrix with coefficients in the polynomial ring  ${\mathbb Z}[x_1,\dots, x_n]$ defined as:
$$M_G:={\rm Diag}(x_1,\dots,x_n)-A_G.$$
Let 
$$d_G(x_1,\dots,x_n):={\rm determinant}(M_G) \in {\mathbb Z}[x_1,\dots, x_n].$$
The matrix $M_G$ is considered in \cite{C-V}, where the principal ideal of ${\mathbb Z}[x_1,\dots, x_n]$ generated by $d_G(x_1,\dots,x_n)$ is called the $n$-th {\it  critical ideal} of the graph $G$. We will call $d_G$ the {\it critical polynomial} of $G$. Since the adjacency matrix $A'_G$ associated with a different ordering of the vertices of $G$ is of the form $A'_G=P^{-1}A_GP$ for some permutation matrix $P$, the  polynomial $d_G$ is indeed independent of the choice of the ordering of the vertices of $G$. The polynomial $d_G$ is also considered in \cite{GRT}, Proposition 1, where it is shown that two simple graphs $G$ and $G'$ on $n$ vertices are isomorphic if and only if there is an ordering of the vertices of $G$ and of the vertices of $G'$ such that the associated polynomials $d_G(x_1,\dots, x_n)$ and $d_{G'}(x_1,\dots, x_n)$
are equal.

When $G$ is a graph, 
it is well-known that the   set $V_{d_G}(1)$ always contains the value $0$. Indeed, recall that the {\it degree}  $d_i$ of the vertex $v_i$ in a graph $G$ is the number of edges of $G$ attached to $v_i$. 
The {\it Laplacian} of the graph is obtained by evaluating $M_G$ at $x_i=d_i$ for $i=1,\dots,n$. 
The Laplacian of $G$ has determinant $0$ since the columns of the Laplacian add to zero, and so are linearly dependent. 
We show in this article in Theorem \ref{thm.main1} (i) that the set $V_{d_G}(1)$  can be completely described, and is always equal to ${\mathbb Z}_{\geq 0}$ . 

Motivated by geometric questions recalled in \ref{rem.why2}, we also investigate in this article the properties of the set $V_{d_G}(2)$.  Consider for instance   the example of 
the path $A_2$ on two vertices. In this case, 
 $$M_{A_2}=\left( \begin{array}{cc}
x_1&-1 \\
-1&x_2
\end{array}
\right)
\mbox{ and }  d_{A_2}(x_1,x_2)=x_1x_2-1.$$
It is clear that $V_{d_{A_2}}(1)={\mathbb Z}_{\geq 0}$. The complement of $V_{d_{A_2}}(2) $ in $V_{d_{A_2}}(1)$ is infinite
since 
$$V_{d_{A_2}}(2)={\mathbb Z}_{ \geq 1} \setminus    \{p-1 \mid p {\rm \ prime}\}.$$

To state our results on $V_{d_G}(2)$, we first recall the following concepts.
 Let $a_1,\dots,a_n \in {\mathbb Z}_{>0}$. Let us denote by $M_G(a_1,\dots,a_n)$ the integer matrix obtained from $M_G$ by evaluating $M_G$ at $x_i=a_i$, $i=1,\dots,n$. Such matrices play a considerable role in geometry, where they might be in addition
endowed with a positivity property.
Recall that a symmetric integer matrix $B \in M_n(\mathbb Z)$ is {\it  positive semi-definite} (resp. {\it positive definite}) if for every non-zero vector $X \in  {\mathbb Z}^n$, 
we have $^tXBX \geq 0$ (resp. $^tXBX >  0$).
When matrices of the form $M_G(a_1,\dots,a_n)$ are positive definite, they are $M$-matrices (\cite{Ple}, Definition).

In some geometric contexts, such as when $M_G(a_1,\dots,a_n)$ is obtained as the intersection matrix associated with a finite set of curves on a surface, the following finite abelian group $ \Phi_{M_G}$ is of interest.
Let   $M \in M_n(\mathbb Z)$. Then
\begin{enumerate}[\rm (a)]
\item If $\det(M)\neq 0$, the {\it discriminant group} $\Phi_M:={\mathbb Z}^n/{\rm Im}(M)$ has order $|\det(M)|$.
\item More generally, when ${\rm rank}(M)=\rho<n$, then ${\mathbb Z}^n/{\rm Im}(M)$  is isomorphic to the product of ${\mathbb Z}^{n-\rho}$
by a finite abelian group that we will denote $\Phi_M$. In other words, $\Phi_M$ is isomorphic to the torsion subgroup of ${\mathbb Z}^{n}/{\rm Im}(M)$.
\end{enumerate}

For instance, given a connected graph $G$, consider its Laplacian  $L=M_G(d_1,\dots,d_n)$.
 The kernel of $L$ is generated by the vector $^t\! R=(1,\dots, 1)$ and the group $\Phi_L$ 
can be identified with the group $\Ker(^t\!R)/{\rm Im}(L)$. Its order is the number of spanning trees of the graph $G$. 
The group $\Phi_L$ is found under various names in the literature (see for instance the introduction to \cite{Lor08}, and \cite{Lor89}, \cite{BHN}, \cite{Big}); in this article, we will call $\Phi_L$ the {\it critical group} of the graph. 

The pair $(L,R)$ attached to $G$ can be generalized as follows.  An {\it arithmetical structure} on $G$ (see \cite{Lor89}, Theorem 1.4) is a pair $(M,R)$ such that $M=M_G(a_1,\dots,a_n)$ for some $ a_1,\dots ,a_n \in {\mathbb Z}_{\geq 1}$
and such that $^t\!R=(r_1,\dots,r_n)$ is an integer vector with positive coefficients and $\gcd(r_1,\dots, r_n)=1$ satisfying $MR=0$.
It turns out that then $M$ is positive semi-definite of rank $n-1$. The associated group $\Phi_M$ is isomorphic to $\Ker(^t\!R)/{\rm Im}(M)$.
Such arithmetical structures arise in algebraic geometry, and much is known about the associated group $\Phi_M$ (see, e.g., \cite{Lor1990}).

Define now a subset $V_{G}(r) \subseteq V_{d_G}(r)$ as follows:
$u \in {\mathbb Z}_{\geq 0}$ belongs to $V_{G}(r)$ if and only if there exists $ a_1,\dots ,a_n \in {\mathbb Z}_{\geq r}$ such that:

\begin{enumerate}[\rm (i)]
\item $u=\det(M_G(a_1,\dots,a_n))$.
\item The matrix $M_G(a_1,\dots,a_n) $ is positive definite when $\det(M_G) \neq 0$, and positive semi-definite of rank $n-1$ when $\det(M_G) = 0$.
\item The associated group $\Phi_{M_G(a_1,\dots, a_n)}$ is cyclic.
\end{enumerate}

Recall that a graph $H$ is an {\it induced subgraph of $G$} if it can be obtained by removing from $G$ 
a non-empty set of vertices of $G$  along with all the edges of $G$ attached to any of these vertices. We are now ready to state our main theorems.

\begin{theorem} {\rm (proved in \ref{proof.{thm.main1}})} \label{thm.main1} 
Let $G$ be a connected graph. Then 
\begin{enumerate}[\rm (a)]
\item 
 $V_{G}(1)= {\mathbb Z}_{\geq 0}$.
\item 
Suppose that $G$ contains an induced connected subgraph $H$ such that $1 \in V_H(2)$.  
Then 
\begin{enumerate}[\rm (i)]
\item $V_{G}(2) $ contains 
${\mathbb Z}_{> 0}$. 
\item If $G$ is obtained from $H$ through a sequence of induced subgraphs $H=H_1 \subset H_2 \subset H_k=G$ such that for each $i=1, \dots, k-1$,  $H_{i+1}$ is constructed from $H_i$ by adding exactly one vertex of degree at least $2$, then $V_{G}(2) ={\mathbb Z}_{\geq  0}$. 
\end{enumerate}
\end{enumerate}
\end{theorem} 
As noted above, we always have $0 \in V_{d_G}(1)$ because the determinant of the Laplacian $L$ of $G$ is $0$.
On the other hand, the critical group $\Phi_L$ associated with  $L$ is not always cyclic. 
The question of determining the proportion of connected graphs having cyclic critical groups was raised for instance in \cite{Lor08}, section 4, and progress on this question for random graphs can be found in \cite{Woo}.
 Part (a) of Theorem \ref{thm.main1} implies  that $0 \in V_G(1)$, and in particular  that on any graph $G$,
there always exists an arithmetical structure $(M,R)$ whose associated group $\Phi_M$ is cyclic (in fact, even trivial, see \ref{pro.trivial}).

 Recall that a graph is {\it simple} if there is at most one edge between any two vertices of $G$. The smallest simple graphs $H$ with $1 \in V_{H}(2)$ both have $4$ vertices: 
the extended cycle ${\mathcal C}_3^+$ (see \ref{ex.lollipop}) and the cone $C(A_3)$ on the path $A_3$ on $3$ vertices (see \ref{ex.cone}). Part (b) of Theorem \ref{thm.main1}  
allows us to prove the following general theorem.

\begin{theorem} \label{thm.exceptions} {\rm (proved in \ref{proof.{thm.main2}})}\label{thm.main2}
Let $G$ be a connected simple graph.
Then $V_G(2) \supset {\mathbb Z}_{>0}$, except possibly if $G$   is  a tree, a cycle  ${\mathcal C}_n$,
a  complete bipartite graph ${\mathcal K}(p,q)$, 
a complete graph ${\mathcal K}_n$, $  n \leq 13$,   an extended cycle ${\mathcal C}_n^+$, $n\leq 7$, or the cone $C(A_3)$. 
\end{theorem}

Whether the value $1$ belongs to $V_{G}(2) $ when $G$ is a tree 
was investigated in \cite{B-D, B-H, B-J, B-V, BJM}. 
Corollary 11 in \cite{B-J} gives a list of trees $H$ such that if a tree $G$ of diameter at least $4$ contains such $H$, then $1 \in V_{G}(2) $. The smallest such trees have $8$ vertices, starting with the Dynkin diagram $E_8$.
 
\if false
Three infinite families of trees  where $1 \notin V_{G}(2)$ are given in \cite{B-D}, and we recall them in Section \ref{sec.Dynkin} on Dynkin diagrams. Two of these families also have $0 \notin V_{d_G}(2)$. It would be interesting to classify the graphs where  $0 \notin V_{d_G}(2)$. Three additional such examples are found in Remark \ref{rem.0notinV}.
\fi
 
Recall that the star ${\mathcal S}_n$ on $n \geq 4$ vertices is  a tree with a  vertex of degree $n-1$. 
It is shown in \cite{B-H}, Proposition 6, that  $1 \notin V_{{\mathcal S}_n}(2) $ when $n\leq 59$.
No integer $n \geq 4$ is known such that $1 \in V_{{\mathcal S}_n}(2) $. 

Denote by ${\mathcal S}_n^+$ the graph obtained by adjoining a single vertex to the star  ${\mathcal S}_n$ and linking it with a single edge to a vertex of degree $1$ in  ${\mathcal S}_n$. 
The family ${\mathcal S}_n^+$, $n \geq 4$, is another family of graphs where none of its members are known to have $1 \in V_{{\mathcal S}^+_n}(2) $.

When Part (b) of Theorem \ref{thm.main1} does not apply,
the  set $V_{G}(2)$ seems very difficult to describe precisely.
An example where the complement of $V_{G}(2)$ in $V_{d_G}(2)$ is infinite is given in \ref{rem.ComplementInfinite}.
Theorem \ref{thm.density} can often be used to prove that, at least, $V_{d_G}(2)$ is {\it dense} in ${\mathbb Z}_{\geq 0}$ (the definition of  dense  in this context is recalled in \ref{def.dense}). We state below an explicit consequence of Theorem \ref{thm.density} which complements Theorem \ref{thm.main2}.
It would be interesting to determine whether  $V_{d_G}(2)$ is always dense in ${\mathbb Z}_{\geq 0}$.

\begin{theorem} {\rm (proved in \ref{proof.{thm.denselist}})} \label{thm.denselist}
Let $G$ be one of the following graphs: 
\begin{enumerate}[\rm (a)]
\item $G$ is a Dynkin diagram,   an extended Dynkin diagram,  a star ${\mathcal S}_n$, or an extended star ${\mathcal S}_n^+$.
\item $G$ is a cycle ${\mathcal C}_n$ or  an extended cycle ${\mathcal C}_n^+$.
\item $G$ is  the cone $C(A_3)$,  the complete graph ${\mathcal K}_n$,   the complete bipartite graph ${\mathcal K}(2,n)$ or ${\mathcal K}(3,n)$.
\end{enumerate}
Then the set $V_{d_G}(2)$ is dense in ${\mathbb Z}_{\geq 0}$.
\end{theorem}
 The only  examples of graphs where the complement of $V_{d_G}(2)$ in ${\mathbb Z}_{\geq 0}$ is {\it known  to be infinite} are the {\it banana graphs}, the graphs on two vertices linked by $e\geq 1$ edges. Thus it is also natural to wonder whether the complement of $V_{d_G}(2)$ in ${\mathbb Z}_{\geq 0}$ is not only of density $0$ but in fact is finite for most graphs.

Our study  was motivated by considerations from algebraic geometry.
In the next remark, we give a brief exposition of how the results of this article pertain to this field of research.

\begin{remark} \label{rem.why2}
Matrices of the form $M_G(a_1,\dots,a_n)$ with $a_1,\dots,a_n \geq 1$
arise in algebraic geometry when considering a finite collection of curves $\{C_i$, $i=1,\dots, n\}$, on a non-singular surface $S$. Attached to each pair of distinct curves $C_i$ and $C_j$ is an {\it intersection number} $(C_i \cdot C_j) \geq 0$ which counts (with multiplicities) how many times $C_i$ and $C_j$ intersect (see, e.g., \cite{Liubook}, 9.1). The {\it dual graph} $G$ attached to the configuration of curves is the graph on $n$ vertices $v_1,\dots, v_n$ such that when $i \neq j$,  $v_i$ is linked to $v_j$ by $(C_i \cdot C_j)$ edges.  

Each curve $C_i$ on $S$  has a {\it self-intersection number} $(C_i \cdot C_i)$, and these numbers are known to be strictly negative when the configuration $\cup_{i=1}^n C_i$ occurs as the exceptional divisor of the resolution of a singularity.  The {\it intersection matrix} $( (C_i \cdot C_j))_{1 \leq i,j\leq n}$ is then of the form $-M_G(a_1,\dots,a_n)$ with $a_1,\dots,a_n \geq 1$, and is known to be negative-definite. It is often the case that a {\it minimal} resolution of singularities leads to a matrix $M_G(a_1,\dots,a_n)$ with $a_1,\dots,a_n \geq 2$, which explains our interest in understanding the possible values of the determinants of such matrices when $a_1,\dots,a_n \geq 2$. 

Matrices $M_G(a_1,\dots,a_n)$ which are only positive semi-definite
arise from configuration of curves  associated with a degeneration of a non-singular curve, 
and in this case   
{\it minimal} special fibers of degenerating curves generally also lead to matrices $M_G(a_1,\dots,a_n)$ with $a_1,\dots,a_n \geq 2$.  

The collection of curves $\{C_i, i=1,\dots, n\}$ attached to the resolution of a surface singularity, and its associated intersection matrix, play an important role in understanding the singularity. It is still an open problem to completely characterize the matrices that can occur as intersection matrices associated to ${\mathbb Z}/p{\mathbb Z}$-quotient surface singularities
in prime characteristic $p$. 
In previous works on
such singularities, 
the author showed that the intersection matrix $M$ associated with the resolution of such quotient singularity
can only have determinant equal to a power of $p$, and that the finite group $\Phi_M$ associated to $M$ is killed by $p$
(see \cite{Lor13}, 3.18, and \cite{L-S}, 6.3, 7.1, for examples). 
The results of this article indicate that matrices $M_G(a_1,\dots,a_n)$ of prime determinant $p$ are plentiful.

Motivated by the problem of classifying resolutions of cyclic quotient singularities,
it is natural to wonder, given a graph $G$,  whether, for all but finitely many primes $p$, there always exists a set of diagonal elements $a_1,\dots,a_n \geq 2$ (depending on $p$) such that $M_G(a_1,\dots,a_n)$ has determinant $p$.
The answer to this question would be positive if it were possible to show, more generally, that the complement of the set $V_G(2)$ in ${\mathbb Z}_{\geq 0}$ is finite.
\end{remark}

\begin{remark} It is a classical problem
in number theory to study the integer values taken by an integer polynomial
$f(x_1,\dots,x_n)$.   When $G$ is a graph, the polynomial $d_G(x_1,\dots, x_n)$ is a polynomial of degree $n$ in $n$ variables consisting only of {\it squarefree} monomials.
A famous polynomial in number theory, $f:=x_1^n+\dots + x^n_n$, is also of degree $n$ in $n$ variables
but is indeed as far as having squarefree monomials as possible. The problem of determining the set $V_{f}(1)$ in this case is  
related to the classical Waring's problem. 
When $n=2$, the set $V_f(1)$ has positive density in ${\mathbb Z}_{\geq 0}$, but does not contain any integer that is congruent to $3$ modulo $4$. 
When $n=3$, the set $V_f(1)$ is infinite, but does not contain any integer that is congruent to $4$ or $5$ modulo $9$. It is an open problem in this case
to determine whether the set $V_f(1)$  has positive density (see \cite{DHL} for positive evidence towards this question).

Let us mention here another analogous question in number theory
where the set  $V_g(1)$ in this case misses only finitely many values, but where it is still an open question to completely determine $V_g(1)$. 
\if false
The following is an example of a polynomial of degree $2$ in three variables where it is possible to show that $MV_{d_G}(1)$ is finite and not empty, but where the complete determination of the set $MV_{d_G}(1)$ is not yet known.
\fi
The 
polynomial $g:=xy+yz+zx$ 
consists only of squarefree monomials. 
When $i>0$ is any integer, the equation   $xy+yz+zx=i$ always has solutions in positive integers except for at most $19$ values of $i$ (\cite{B-C} Theorem 1.1).
The first $18$ such values are known explicitly and are in the interval $[1,462]$. If the Generalized Riemann Hypothesis is assumed, the complement of the set $V_{g}(1)$ in ${\mathbb Z}_{>0}$ consists exactly of these $18$ known values.
\end{remark}

This article exhibits many graphs $G$ where the 
set $V_{d_G}(2)$
misses some positive values, but computations nevertheless suggest that it contains all positive values except for finitely many (see, e.g., \ref{ex.lollipop}, \ref{ex.Dynkin4}, \ref{ex.paths}).
It would be interesting to determine if these polynomials $d_G(x_1,\dots,x_n)$ indeed have this property.
The easiest example of such polynomial is  $d_{A_3}(x,y,z)=xyz-x-z$, associated with
the path $A_3$ on $3$ vertices. Computations suggest that the complement of the set $V_{A_3}(2)$ in ${\mathbb Z}_{\geq 0}$  is contained
in the set $[ 0,1, 2, 3, 5, 6, 9, 11, 14, 15, 35, 105, 510 ]$ 
(see  Proposition \ref{pro.density4}).
\end{section}

 \begin{section}{First Main Theorem}
 Let $M \in M_n({\mathbb Z})$. We will use the following standard notation. 
 Let $M_{ij}$ denote the submatrix
obtained  by removing from $M$ its $i$-th row and its $j$-th column. 
Let $M^*$ denote the {\it adjoint} of $M$, with $MM^*=(M^*)M=\det(M) {\rm Id}_n$. 
By definition, $(M^*)_{ij}=(-1)^{i+j} \det(M_{ji})$.
The matrix $M$ is a {\it positive} matrix if all the entries of $M$ are positive.

The group $\Phi_M$ is isomorphic by definition to the torsion subgroup of ${\mathbb Z}^n/{\rm Im}(M)$. If $0<{\rank}(M)=\rho<n$, then there exist two matrices $P$ and $Q$ in ${\rm GL}_n({\mathbb Z})$ such that
$PMQ$ is a diagonal matrix of the form ${\rm Diag}(0,\dots, 0,f_1,\dots, f_{\rho})$, with $f_1 \mid f_2 \mid \dots \mid f_{\rho}$. This diagonal matrix is called the Smith Normal Form of $M$.
 The group $\Phi_M$ is isomorphic to $\prod_{i=1}^\rho {\mathbb Z}/f_i{\mathbb Z}$, and thus  $\Phi_M$ is cyclic if and only if $\rho=1$ or $f_{\rho-1}=1$. 

\begin{emp} \label{emp.classical}
Let $G$ be a connected graph on $n$ vertices. The matrices $M=M_G(a_1,\dots,a_n)\in M_n({\mathbb Z})$ with $a_1,\dots, a_n \geq 1$ have several very useful properties when they are positive semi-definite.

\begin{enumerate}[\rm (a)]
\item Assume that $\det(M) \neq 0$ and that $M$ is positive definite. Then the inverse $M^{-1}$ of $M$ is a positive matrix. 
\item  Assume that $\det(M) = 0$ and that $M$ is positive semi-definite of rank $n-1$. Then there exists a unique vector  $R$ in ${\mathbb Z}_{>0}^n$
with coprime coefficients and such that $MR=0$. We have $M^*=|\Phi_M| R (^t\! R)$. 
\item 
Assume that $M_G(a_1,\dots,a_n)$ is as in (a) or (b). For any non-zero vector $(b_1,\dots, b_n) \in {\mathbb Z}_{\geq 0}^n$, the matrix $M_G(a_1+b_1,\dots,a_n+b_n)$ is positive definite,
and $$\det M_G(a_1,\dots,a_n) < \det M_G(a_1+b_1,\dots,a_n+b_n).$$
\end{enumerate}
\end{emp}
Property (a) is F15, page 180 of \cite{Ple}. Property (b) follows from Proposition 1.1 and Theorem 1.4 in \cite{Lor89}. Property (c) is A3, page 179 of \cite{Ple}, when $M_G(a_1,\dots, a_n)$ is positive definite. If it is only positive semi-definite, show first that $M_G(a_1,\dots, a_i+1,\dots, a_n)$ is definite positive for any $i$, and apply A3 to these $n$ positive definite matrices.

\begin{remark}
When $M=M_G(a_1,\dots,a_n)$ is positive definite as in (a), there exists a unique positive vector $R$ minimal with the property  that $MR$ is positive (\cite{Art}, page 132). This vector is called the {\it fundamental vector} of the matrix $M$. The quantity $
(^tR)MR$ is an important numerical invariant associated with $M$. When $M=M_G(a_1,\dots,a_n)$ is positive semi-definite of rank $n-1$,
numerical invariants associated with the arithmetical structure $(M,R)$ described in (b) are discussed in \cite{Lor12}, 2.1 and 4.1.
\end{remark}

\begin{remark} \label{rem.finiteness}
\if false
 Let $w \in V_{d_G}(1)$. It may happen that there are infinitely many points
 $ (a_1,\dots ,a_n) \in {\mathbb Z}^n_{\geq 1}$ such that $\det M_G(a_1,\dots,a_n)=w$  
 when 
 $M_G(a_1,\dots,a_n)$ is not assumed to be positive  semi-definite.
 For instance, consider the graphs $G_0$ and $G_1$ below. We have marked each vertex with the corresponding 
 coefficient of a diagonal $(a_1,\dots, a_n)$, such that for any values of $x,y$, we have $\det M_{G_0}(a_1,\dots, a_n)=0$ and $\det M_{G_1}(a_1,\dots, a_n)=1$:
$$\begin{array}{cc}
G_0 \quad
\begin{tikzpicture}
[node distance=0.8cm, font=\small] 
\tikzstyle{vertex}=[circle, draw, fill, inner sep=0mm, minimum size=0.8ex]
\node[vertex]	(v1)  	at (0,0) 	[label=below:{1}] 		{};
\node[vertex]	(v2)		[right of=v1, label=below:{$x$}]	{};
\node[vertex]	(v3)			[right of=v2, label=below:{1}]	{};
\node[vertex]	(v4)			[right of=v3, label=below:{1}]	{};
\node[vertex]	(v6)			[above of=v2, label=above:{1}]	{};
 
\draw [thick] (v1)--(v4);
\draw [thick] (v2)--(v6);
\draw[thick] (v1)--(v6);
\end{tikzpicture}
& \quad \quad \quad 
 G_1 
 \quad
\begin{tikzpicture}
[node distance=0.8cm, font=\small] 
\tikzstyle{vertex}=[circle, draw, fill, inner sep=0mm, minimum size=0.8ex]
\node[vertex]	(v1)  	at (0,0) 	[label=below:{1}] 		{};
\node[vertex]	(v2)		[right of=v1, label=below:{$x$}]	{};
\node[vertex]	(v3)			[right of=v2, label=below:{$1$}]	{};
\node[vertex]	(v4)			[right of=v3, label=below:{1}]	{};
\node[vertex]	(v5)			[above of=v3, label=above:{$y$}]	{};
\node[vertex]	(v6)			[above of=v2, label=above:{1}]	{};
 
\draw [thick] (v1)--(v4);
\draw [thick] (v3)--(v5);
\draw [thick] (v2)--(v6);
\draw[thick] (v1)--(v6)--(v5);
\end{tikzpicture}
\end{array}
$$

 Let now $w \in V_{G}(1)$. With this added assumption of semi-definiteness, 
 \fi
 Let $w \in {\mathbb Z}_{\geq 0}$.
 It is known  that
 there are only {\it finitely many} points
 $ (a_1,\dots ,a_n) \in {\mathbb Z}^n_{\geq 1}$ such that $M_G(a_1,\dots,a_n)$ is positive semi-definite and has determinant $w$. This statement is proved in \cite{L-N}, Theorem 1, when $w>0$ and the matrix is positive definite, and in \cite{Lor89}, 1.6, when $w=0$ and the matrix is positive semi-definite of rank $n-1$.
 
  Counting explicitly the number of distinct arithmetical structures on certain graphs is addressed for instance in \cite{bident}, \cite{BCCGGKMMV},  and \cite{C-V18}. Counting the number of solutions to $d_G(x_1,\dots,x_n)=1$ when $G$ is the path $A_n$ is found in \cite{L-N}.
 \end{remark}

Our next proposition shows that the existence of an arithmetical structure on $G$ implies
that infinitely many values in $V_{d_G}(1)$ are known explicitly. We denote by $\kappa$  the {\it number of spanning trees} of a graph $G$. 
 
 \begin{proposition} \label{pro.complexity} \label{pro.arithmeticalstructure}
 Let $G$ be a connected graph on $n$ vertices.

\begin{enumerate}[\rm (a)]
\item
Suppose that 
  every vertex in $G$ has degree at least $d$.
  Then $V_{d_G}(d)$ contains all positive multiples of $\kappa$. Moreover,
for each $\ell >0$, there exists a positive definite matrix $M_G(a_1,\dots, a_n)$ with $a_i \geq d$ such that
$d_G(a_1,\dots, a_n)= \ell \kappa$. 
\item
More generally, let $(M,R)$ be any arithmetical structure on $G$. Write  $^t\! R=(r_1,\dots,r_n)$ with $\gcd(r_1,\dots, r_n)=1$, $M=$ ${\rm Diag}(a_1,\dots,a_n)-A_G$,  and let $\Phi_M$ denote the associated group. Let $a_{min}$ denote the minimum of the integers $a_1,\dots, a_n$. 
Then $V_{d_G}(a_{min})$ contains every integer of the form $\ell|\Phi_M| r_i^2$ for any integer $\ell \geq 0$ and any $i=1,\dots, n$. 
\end{enumerate}
\end{proposition}
\begin{proof}
(a) Recall that $M_G(d_1,d_2,\dots, d_n) $ is the Laplacian of $G$. 
It is well-known that the determinant of any principal submatrix of size $n-1$  of the Laplacian is equal to $\kappa$.
Consider the matrix $M_G(t,d_2,\dots, d_n) \in {\mathbb Z}[t]$. 
Its determinant is $\kappa(t -d_1)$. Indeed, it is clear that this determinant is a linear polynomial in $t$. The coefficient of $t$ is $\kappa$, and $t=d_1$ must be a root of the polynomial.

 For every value $\ell+d_1 > d_1 \geq d$, we have $M_G(\ell+d_1,d_2,\dots, d_n) $
positive definite since  $M_G(d_1,d_2,\dots, d_n) $ is positive semi-definite of rank $n-1$ (\ref{emp.classical} (c)).

(b) Recall that for an arithmetical structure $(M,R)$, we have 
$M^*=|\Phi|R(^t\!R)$ (\ref{emp.classical} (b)).  Consider the matrix $M_G(t,a_2,\dots,a_n)$. 
Its determinant is $|\Phi_M|r_1^2(t-a_1)$. For every value $\ell+a_1 > a_1 \geq a_{min}$, we have $M_G(a_1+\ell,a_2,\dots, a_n) $
positive definite  of determinant $|\Phi_M|r_1^2\ell$.
\end{proof}

\begin{remark} \label{rem.ComplementInfinite} Let $M_G(a_1,\dots, a_n)$ be a positive definite matrix  with $a_i \geq 2$ such that
$d_G(a_1,\dots, a_n)= \ell \kappa$, as in Proposition \ref{pro.arithmeticalstructure} (a) with $d=2$. It is not always possible to find such matrix such that its associated group $\Phi_{M_G}$ is cyclic. Indeed, in the case of the cycle ${\mathcal C}_2$ on $n=2$ vertices, which has $\kappa=2$, the matrix
$$M_{{\mathcal C}_2}(x,y)= \left( \begin{array}{cc}
x & -2 \\
-2 & y
\end{array}
\right)
$$
has determinant $xy-4$. When $xy-4=2\ell$, and $2\ell+4$ is a power of $2$ and  $x,y \geq 2$,
this equation  has only solutions  with both $x$ and $y$ even.
In this case, $\ell$ is even, and the associated group $\Phi={\mathbb Z}/2{\mathbb Z} \times {\mathbb Z}/\ell{\mathbb Z}$ is not cyclic.

In particular, $G={\mathcal C}_2$ is an example where 
the complement of $V_G(2)$ in $V_{d_G}(2)$ is infinite, since this complement contains every integer of the form $2\ell$ with $\ell=2^m-2$.

\if false
$$ \left( \begin{array}{cc}
\ell+2 & -2 \\
-2 & 2
\end{array}
\right)
$$
has determinant $2\ell$. When $\ell$ is even, the associated group $\Phi={\mathbb Z}/2{\mathbb Z} \times {\mathbb Z}/\ell{\mathbb Z}$ is not cyclic.
\fi
\end{remark}

\begin{corollary} Let $G$ be a tree. Then $V_{d_G}(1) ={\mathbb Z}_{\geq 0}$.
\end{corollary}
\begin{proof} The corollary follows immediately from Proposition \ref{pro.complexity} (a) since 
$\kappa=1$ when $G$ is a tree. 
\end{proof}

The following lemma is needed in the proof of our next proposition.
\begin{lemma} \label{lem.matrix}
Let $N$ denote a $n-1 \times n-1$ square 
matrix with coefficients in a commutative ring $A$. 
Let $M$ denote the following $n\times n$ matrix in $A[t,t_2,\dots,t_n]$:
$$M = \left( \begin{array}{cccc}
t   & t_2 & \cdots & t_n \\
t_2 &     &       &   \\
\vdots & & N & \\
t_n & & & 
\end{array} 
\right)
.$$
Let $^t  T:=(t_2,\dots, t_n)$.
Then 
$$\det(M)= \det(N)t - (^t  T) (N^*) T.$$ 
\end{lemma}
\begin{proof}
Recall that by definition, $(N^*)_{ij} = (-1)^{i+j}\det(N_{ji})$.
The lemma follows by expanding $\det(M)$ using the first row of $M$.
\end{proof}

 Given a vertex $v$ of $G$, let $G_v$ denote the subgraph of $G$ obtained  by removing from $G$
the vertex $v$ and all the edges attached to $v$.

\begin{proposition} \label{pro.induction}
Let $G$ be a connected graph. Let $v$ be a vertex such that 
$G_v$ is connected and $1 \in V_{G_v}(r)$. Then 
\begin{enumerate}[\rm (a)]
\item $V_G(r) \supseteq {\mathbb Z}_{>0}$ when $r=1$ or $2$. 
In general, $V_G(r) \supseteq {\mathbb Z}_{\geq r-1}$.
\item There exists on $G$ an arithmetical structure such that the associated group $\Phi$ is trivial and, hence, cyclic. In particular, $0 \in V_G(1)$.
\item If the degree of $v$ is at least $2$ and $1 \in V_{G_v}(2)$, then $0 \in V_G(2)$. More precisely, there exists 
on $G$ an arithmetical structure $(M,R)$ with 
$M=M_G(a_1,\dots, a_n)$ and $a_1,\dots, a_n \geq 2$ such that the associated group $\Phi_M$ is trivial.
\end{enumerate}
\end{proposition}
\begin{proof}
Without loss of generality, we may assume that $v=v_1$.
By hypothesis we can find $a_2, \dots, a_n \geq r$
such that the matrix $N:=M_{G_v}(a_2,\dots, a_n)$ has determinant $1$
and  is positive definite. 
Consider then the determinant
of the matrix $M_G(t,a_2,\dots, a_n)$, which has $
N$
in its lower right corner. 
Since $N$ is positive definite by hypothesis, we find that
$N^*$ is a positive matrix (use \ref{emp.classical} (a)). Hence,  Lemma \ref{lem.matrix} applied to $M_G(t,a_2,\dots, a_n)$
shows that  $$\det M_G(t,a_2,\dots, a_n)= t-a$$ with $a>0$.

(a) 
Let $a_1 > a$. Since  
$N$ is positive definite and $\det M_G(a_1,a_2,\dots, a_n)>0$, we find that
the matrix  $M_G(a_1,a_2,\dots, a_n)$ is  positive  definite.
In addition, its associated group $\Phi$ is cyclic. Indeed, 
$M_G(a_1,a_2,\dots, a_n)$ contains the square submatrix $N$
of size $n-1 \times n-1$, which has determinant $1$. 
It is well-known that the group $\Phi$ is cyclic if and only if 
the greatest common divisor of all the minors of size $n-1$ is equal to $1$.

Suppose  that $r \geq a+1$.
 Then $V_G(r) \supseteq {\mathbb Z}_{\geq r-a}$.
Suppose now that $r \leq a+1$. 
Then  $V_G(r) \supseteq {\mathbb Z}_{\geq 1}$.  This proves (a).

(b) Consider now the matrix $M:=M_G(a,a_2,\dots, a_n)$, of determinant $0$. 
We claim that this matrix is positive semi-definite. We prove this claim by exhibiting a positive vector $R \in {\mathbb Z}^n$, of the form $^t\! R=(1, ^t \! R_0)$,
with $MR=0$. Recall that the matrix $M$ has the form 
$$M = \left( \begin{array}{cccc}
a   & -s_2 & \cdots & -s_n \\
-s_2 &     &       &   \\
\vdots & & N & \\
-s_n & & & 
\end{array} 
\right)
,$$
for some non-negative integers $s_2,\dots,s_n$.
Since the graph $G$ is connected, one at least of the integers $s_2,\dots,s_n$ must be positive.
Write $^t\!S=(s_2,\dots, s_n)$.  
Since $\det(N)=1$, we can find an integer vector $R_0$ such that $-S+NR_0=0$. Since $S$ is a non-negative vector and $N^*$ is a positive matrix, we find that $R_0=N^*S$ is a positive vector. Lemma \ref{lem.matrix}
shows that $\det(N)a=(^t \! S) N^* S$. It follows that $MR=0$, as desired. Now suppose that $X \in {\mathbb Z}^n$. To show that $^tXMX \geq 0$, 
we note that it is always possible to write $X=X_0+\alpha R$, with 
$X_0\in {\mathbb Z}^n$ of the form $^t\!X_0=(0,^t\!X_1)$ and $\alpha \in {\mathbb Z}$. Therefore, $^t\!XMX = \ ^t\!X_1NX_1 \geq 0$ since $N$ is positive definite by construction.

Since $\det(N)=1$ and the first coefficient of $R$ is $1$, we find that the group $\Phi_M$ associated with $M$ is trivial (use \ref{emp.classical} (b)) and, hence, cyclic. This shows that $0 \in V_G(1)$.

(c) Consider again the structure $(M,R)$ introduced in (b). It is clear that when $1 \in V_{G_v}(2)$ and 
$a \geq 2$, then $0 \in V_G(2)$. The integer $a$ is obtained as $a=(^t \! S) N^* S$. The hypothesis that the degree of $v$ is at least $2$ implies that 
if $v$ is linked to only one vertex $w$ of $G_v$, then the number of edges between $v$ and $w$ is at least $2$. 
Thus, since the matrix $N^*$ is positive, we must have 
$a=(^t \! S) N^* S \geq 2$.
\end{proof}

\begin{emp} \label{proof.{thm.main1}}
{\it Proof of Theorem  \ref{thm.main1}.}

In Part (a), the graph $G$ contains a connected induced subgraph $H$ on two vertices linked by $a \geq 1$ edges. The polynomial $d_H(x,y)=xy-a^2$ takes all non-negative values when $x,y \geq 1$. In particular, the value $1$ is taken with $x=a^2+1$ and $y=1$ and so, $1 \in V_H(1)$.

Let now $H$ be any connected induced subgraph of $G$.
Let $w_1,\dots, w_k$ denote the vertices of $G$ that do not belong to $H$. 
Since $G$ is connected, at least one of the vertices $w_1,\dots, w_k$ 
is connected by an edge with a vertex of $H$. Up to renumbering
$w_1,\dots, w_k$, we may assume that $w_1$ is connected by an edge with a vertex of $H$. Let $H_1$ denote the connected induced subgraph of $G$ on the vertices of $H$ and $w_1$. Repeating this procedure, we can assume that we have a chain of connected induced subgraphs
$H \subset H_1 \subset \dots \subset H_k=G$, where 
$H_j$ is the (connected) induced subgraph of $G$ on the vertices of $H_{j-1}$ and $w_j$.

Assume that $1 \in V_H(r)$, with $r =1$ or $2$.  
Then Proposition \ref{pro.induction} (a) can be applied successively to each pair $H \subset H_1$, $H_1 \subset H_2$, $\dots$, $H_{k-2} \subset H_{k-1}$ to show that $1 \in V_{H_{k-1}}(r)$.

When $1 \in V_{H_{k-1}}(1)$, we  use Proposition \ref{pro.induction} (a) and (b) to conclude that $  V_{H_{k}}(1)= {\mathbb Z}_{\geq 0}$. This proves 
Theorem  \ref{thm.main1} (a).

When $1 \in V_{H_{k-1}}(2)$, we  use Proposition \ref{pro.induction} (a)  to conclude that 
$  V_{H_{k}}(2)\supseteq {\mathbb Z}_{> 0}$. 
We then use Proposition \ref{pro.induction} (c)
to conclude the proof of
Theorem  \ref{thm.main1} (b).
\qed
\end{emp}

\begin{corollary} \label{pro.trivial} Let $G$ be a connected graph. Then there exists on $G$ an arithmetical structure $(M,R)$ such that its associated group $\Phi_M$ is trivial.
\end{corollary}

\begin{proof} In the proof of Theorem  \ref{thm.main1} (a) above in \ref{proof.{thm.main1}}, we find that the graph $G$ is such that  $1 \in V_{H_{k-1}}(1)$.
 Proposition \ref{pro.induction} (b)  applied to $H_{k-1} \subset G$ immediately implies the corollary.
\end{proof}

\begin{remark} While $0 \in V_G(1)$, it may happen that $0 \notin V_G(2)$. Indeed, 
$0 \notin V_{d_G}(2)$ when $G$ is a Dynkin diagram (Proposition \ref{pro.Dynkin}). Three additional such examples (the extended cycles ${\mathcal C}_n^+$ for $n=2,3,5$) are found in Remark \ref{rem.0notinV}, and it would be interesting to classify the graphs where  $0 \notin V_{d_G}(2)$.

It is also possible to have $0 \in V_{d_G}(2)$ but $0 \notin V_{G}(2)$. Indeed, this happens for instance for the extended Dynkin diagram $\tilde{D_n}$ with $n$ even (\ref{def.Dynkin}): The graph $\tilde{D_n}$ has only  one arithmetical structure, with matrix $M_{\tilde{D_n}}(2,\dots, 2)$,
and its associated group $\Phi$ is ${\mathbb Z}/2{\mathbb Z} \times {\mathbb Z}/2{\mathbb Z} $ when $n$ is even. 

One may wonder whether it is possible to have $0 \in V_{G}(2)$, but no arithmetical structure $M_G(a_1,\dots,a_n)$
with $a_1,\dots, a_n \geq 2$ such that the associated group $\Phi $ is trivial. One such example might be ${\mathcal C}_6^+$ (see Remark \ref{rem.0notinV}). Another example might be the bipartite graph ${\mathcal K}(4,4)$.
\end{remark}

\begin{remark} Modified appropriately, the proof of Theorem \ref{thm.main1} does produce some information on the set $V_G(3)$ when $1 \in V_H(3)$. We do not investigate the properties of the sets $V_G(r)$ any further in this article when $r\geq 3$.  
\end{remark} 
\end{section}

\begin{section}{Second Main Theorem}

Let us denote by ${\mathcal C}_n^+$ the graph on $n+1$ vertices obtained by attaching a single vertex
using a single edge to the cycle ${\mathcal C}_n$ on $n$ vertices. Such graph is sometimes called a {\it pan}, and  is a type of {\it tadpole} graph.

\begin{example}  \label{ex.lollipop}
Let $G={\mathcal C}_n^+$. 
 In the table below, 
we provide a set $L$ which computations indicate contains the complement of $V_{d_G}(2)$ in ${\mathbb Z}_{\geq 0}$.

\begin{center}
    \begin{tabular}{|c|l|}  
      \hline
      $n$ & $L$ \\
      \hline
      $2$ & [0] (see 
      \ref{rempro.G(a,b)}, graph $A_3(2,1)$) \\
      \hline
      $3$ & [0, 2, 14, 20, 26, 38, 44, 68, 254] \\
      \hline
      $4$ & [2, 3, 7, 10, 19, 39, 79, 154 ] \\
      \hline
      $5$ & [0, 2, 8, 12, 18 ] \\
      \hline
      $6$ & empty \\
      \hline
      $7$ & [6, 66, 94 ]  \\
      \hline    
    \end{tabular}
  \end{center}
  
  \smallskip
  We note here the fact that $1 \in V_G(2)$. Indeed, with the choice of labeling used below, we find that 
$\det M_G(t,2,\dots,2)=(n+1)t-(3n+2)$:
$$ 
\begin{tikzpicture}
[node distance=0.8cm, font=\small] 
\tikzstyle{vertex}=[circle, draw, fill, inner sep=0mm, minimum size=0.7ex]
\node[vertex]	(v1)  	at (0,0) 	[label=above:{$t$}] 		{};
\node[vertex]	(v22)		[right of=v1, label=above:{$2$}]	{};
\node[vertex]	(v33)			[right of=v22, label=above:{$2$}]	{};
\node[vertex]	(w1)			[below of=v1,label=below:{$2$}]	{};
\node[vertex]	(w2)			[right of=w1, label=below:{$2$}]	{};
 
\draw [thick] (v1)--(v22)--(v33) ;
\draw [thick] (v1)--(w1);
\draw[thick, dotted] (w1)--(w2);
\draw[thick, dotted] (v22)--(w2);
\end{tikzpicture}
$$
Thus, choosing $t=3$ produces $1 \in V_{d_G}(2)$ when $G={\mathcal C}_n^+$. To show that $1 \in V_G(2)$, it suffices to note that 
$M_G(3,2,\dots, 2)$ is positive definite. This follows from the fact that removing the first row and first column of the matrix $M_G$ leaves a submatrix equal to the Dynkin diagram $A_n$. This submatrix is known to be positive definite. Hence, Silvester's criterion shows that $M_G(3,2,\dots, 2)$ is positive definite. 
  \end{example}
  \begin{proposition} \label{cor.tadpole}
   Let $G={\mathcal C}_n^+$. Then $1 \in V_G(2)$, and if $n \geq 8$, then 
   $V_G(2) = {\mathbb Z}_{\geq 0}$. 
  \end{proposition}
\begin{proof} We showed  in Example \ref {ex.lollipop} that $1 \in V_G(2)$ when $n \leq 7 $.
The tree $E_8$ 
$$
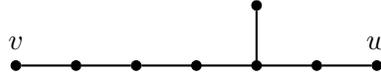
\begin{tikzpicture}
[node distance=0.8cm, font=\small] 
\tikzstyle{vertex}=[circle, draw, fill, inner sep=0mm, minimum size=0.7ex]
\node[]	(v1)  	at (0,0) 	[label=below:{}] 		{};
\node[vertex]	(v22)		[right of=v1, label=above:{$v$}]	{};
\node[vertex]	(v33)			[right of=v22, label=below:{}]	{};
 
\node[vertex]	(v44)			[right of=v33, label=below:{}]	{};
\node[vertex]	(v2)			[right of=v44, label=below:{}]	{};
\node[vertex]	(v3)			[right of=v2, label=below:{}]	{};
\node[vertex]	(v)			[above of=v3,  label=right:{}]	{};
\node[vertex]	(v4)			[right of=v3, label=below:{}]	{};
\node[vertex]	(v5)			[right of=v4, label=above:{$w$}]	{};
\draw [thick] (v22)--(v33)--(v44)--(v4);
\draw [thick] (v)--(v3);
\draw[thick] (v4)--(v5);
\end{tikzpicture}
$$
has $1 \in V_{E_8}(2)$ (see \ref{pro.Dynkin}). 
The graph $G={\mathcal C}_8^+$ is obtained by attaching a new vertex $v_0$ to both $v$ and $w$ with one edge. 
To get all graphs ${\mathcal C}_n^+$ with $n>8$, we first lengthen the chain at $v$, so that the resulting graph 
has $n-1$ vertices, and then add a vertex $v_0$ as above.
 Theorem \ref{thm.main1} (b) then implies that   
   $V_G(2) \supseteq  {\mathbb Z}_{> 0}$. 
To prove that $0\in V_G(2)$, we use the following lemma.
\end{proof}

\begin{lemma} \label{ex.arith.struct.cycle+}
Let $G={\mathcal C}_n^+$ with $n \geq 7$. Then $G$ has the following arithmetical structure $(M,R)$.
There are $k\geq 0$ white vertices in the graphs below. The vertices of the left graph are adorned with the corresponding coefficient of the diagonal of $M$, and  the vertices of the right graph are adorned with the corresponding coefficient of $R$:
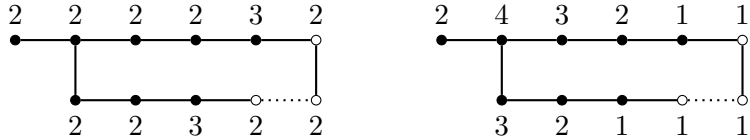
$$
\begin{array}{cc}
\begin{tikzpicture}
[node distance=0.8cm, font=\small] 
\tikzstyle{vertex}=[circle, draw, fill, inner sep=0mm, minimum size=0.7ex]
\node[vertex]	(v1)  	at (0,0) 	[label=above:{$2$}] 		{};
\node[vertex]	(v22)		[right of=v1, label=above:{$2$}]	{};
\node[vertex]	(v33)			[right of=v22, label=above:{$2$}]	{};
 \node[vertex]	(v44)			[right of=v33, label=above:{$2$}]	{};
\node[vertex]	(v2)			[right of=v44, label=above:{$3$}]	{};
\node[vertex, fill=none]	(v3)	[right of=v2, label=above:{$2$}]	{};

\node[vertex]	(w1)			[below of=v22,  label=below:{$2$}]	{};
\node[vertex]	(w2)			[right of=w1, label=below:{$2$}]	{};
\node[vertex]	(w3)			[right of=w2, label=below:{$3$}]	{};
\node[vertex, fill=none]	(w4)			[right of=w3, label=below:{$2$}]	{};
\node[vertex, fill=none]	(w5)			[right of=w4, label=below:{$2$}]	{};
\draw [thick] (v1)--(v22)--(v33)--(v44)--(v2);
\draw [thick] (v22)--(w1);
\draw[thick] (w1)--(w4);
\draw[thick] (w5)--(v3)--(v2);
\draw[thick, dotted] (w4)--(w5);
\end{tikzpicture}
&
\quad \quad

\begin{tikzpicture}
[node distance=0.8cm, font=\small] 
\tikzstyle{vertex}=[circle, draw, fill, inner sep=0mm, minimum size=0.7ex]
\node[vertex]	(v1)  	at (0,0) 	[label=above:{$2$}] 		{};
\node[vertex]	(v22)		[right of=v1, label=above:{$4$}]	{};
\node[vertex]	(v33)			[right of=v22, label=above:{$3$}]	{};
 \node[vertex]	(v44)			[right of=v33, label=above:{$2$}]	{};
\node[vertex]	(v2)			[right of=v44, label=above:{$1$}]	{};
\node[vertex, fill=none]	(v3)	[right of=v2, label=above:{$1$}]	{};

\node[vertex]	(w1)			[below of=v22,  label=below:{$3$}]	{};
\node[vertex]	(w2)			[right of=w1, label=below:{$2$}]	{};
\node[vertex]	(w3)			[right of=w2, label=below:{$1$}]	{};
\node[vertex, fill=none]	(w4)			[right of=w3, label=below:{$1$}]	{};
\node[vertex, fill=none]	(w5)			[right of=w4, label=below:{$1$}]	{};
\draw [thick] (v1)--(v22)--(v33)--(v44)--(v2);
\draw [thick] (v22)--(w1);
\draw[thick] (w1)--(w4);
\draw[thick] (w5)--(v3)--(v2);
\draw[thick, dotted] (w4)--(w5);
\end{tikzpicture}
\end{array}
$$
The associated group $\Phi$ is cyclic of order $2k+5$.
\end{lemma}
\begin{proof} We leave it to the reader to verify that $(M,R)$ is an arithmetical structure. 
Consider the submatrix $M'$ of $M$ obtained by removing the row and column corresponding to the unique vertex $v$ of degree $3$.
 Its determinant is $16(2k+5)$. This shows that $|\Phi|=2k+5$ since the coefficient of $R$ corresponding to $v$ is $4$ and $M^*=|\Phi| R(^t\! R)$ (see \ref{emp.classical} (b)).
 
To show that $\Phi$ is cyclic, it suffices to compute the determinant of a well-chosen $n-2 \times n-2$ submatrix, and show that it is coprime to $|\Phi|$. For this, one can use the submatrix of $M'$ where the row and column corresponding to a vertex of degree $2$ adjacent to $v$ have been removed. 
Its determinant is $2(12k+29)$. We leave the details to the reader.
\end{proof}

\begin{remark} \label{rem.0notinV}
 When $n=4$ and $6$, the following are arithmetical structures on $G={\mathcal C}_n^+$, 
with groups $\Phi$ of order $1$ and $3$ (the vertices are adorned with the corresponding elements on the diagonal of the matrix). This shows that $0 \in V_G(2)$ when  $G={\mathcal C}_4^+$ and ${\mathcal C}_6^+$.
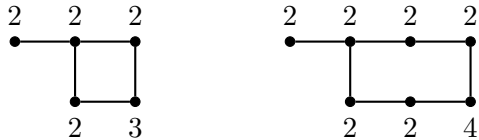
$$\begin{array}{cc}
\begin{tikzpicture}
[node distance=0.8cm, font=\small] 
\tikzstyle{vertex}=[circle, draw, fill, inner sep=0mm, minimum size=0.7ex]
\node[vertex]	(v1)  	at (0,0) 	[label=above:{$2$}] 		{};
\node[vertex]	(v22)		[right of=v1, label=above:{$2$}]	{};
\node[vertex]	(v33)			[right of=v22, label=above:{$2$}]	{};

\node[vertex]	(w1)			[below of=v22,  label=below:{$2$}]	{};
\node[vertex]	(w2)			[right of=w1, label=below:{$3$}]	{};
 
\draw [thick] (v1)--(v22)--(v33);
\draw [thick] (v22)--(w1);
\draw[thick] (w2)--(v33);
 \draw[thick] (w2)--(w1);
 
\end{tikzpicture}

& \quad \quad \quad

\begin{tikzpicture}
[node distance=0.8cm, font=\small] 
\tikzstyle{vertex}=[circle, draw, fill, inner sep=0mm, minimum size=0.7ex]
\node[vertex]	(v1)  	at (0,0) 	[label=above:{$2$}] 		{};
\node[vertex]	(v22)		[right of=v1, label=above:{$2$}]	{};
\node[vertex]	(v33)			[right of=v22, label=above:{$2$}]	{};
 \node[vertex]	(v44)			[right of=v33, label=above:{$2$}]	{};
\node[vertex]	(w1)			[below of=v22,label=below:{$2$}]	{};
\node[vertex]	(w2)			[right of=w1, label=below:{$2$}]	{};
\node[vertex]	(w3)			[right of=w2, label=below:{$4$}]	{};
 
\draw [thick] (v1)--(v22)--(v33)--(v44);
\draw [thick] (v22)--(w1);
\draw[thick] (w1)--(w3);
\draw[thick] (w3)--(v44);
 
\end{tikzpicture}
\end{array}
$$
It is likely that $0 \notin V_{d_G}(2)$ when $G={\mathcal C}_3^+$ and ${\mathcal C}_5^+$. For ${\mathcal C}_2^+$, this statement is proved in 
Proposition \ref{pro.G(a,b)}. Preliminary computations did not find any   arithmetical structure $(M,R)$ on ${\mathcal C}_4^+$ and ${\mathcal C}_6^+$ where all coefficients of the diagonal of $M$ are at least $2$, other than the ones given above.
\end{remark}

\begin{example} \label{ex.cone}
Consider   the cone $G=C(A_3)$ on the path $A_3$ (sometimes called the {\it diamond} graph).
Computations suggest that the complement of $V_G(2)$ in ${\mathbb Z}_{\geq 0}$  is finite, and contained in the set $L:= [ 5, 17, 29, 71, 77, 101, 137, 551]$.
Note that $\det M_G(2,2,5,3)=1$, and the graph $G$ below is adorned with the corresponding coefficients of the diagonal.
$$
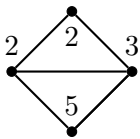
\begin{tikzpicture}
[node distance=0.8cm, font=\small] 
\tikzstyle{vertex}=[circle, draw, fill, inner sep=0mm, minimum size=0.7ex]
\node[vertex]	(v1)  	at (0,0) 	[label=above:{$2$}] 		{};
\node[]	(dummy)  	[right of=v1, label=above:{}] {};
\node[vertex]	(v22)		[right of=dummy, label=above:{$3$}]	{};
\node[vertex]	(v33)			[below of=dummy, label=above:{$5$}]	{};
\node[vertex]	(w1)			[above of=dummy,  label=below:{$2$}]	{};
\draw [thick] (v1)--(v22)--(v33);
\draw [thick] (v1)--(w1)--(v22);
\draw [thick] (v1)--(v33)--(v22);
\end{tikzpicture}
$$ 
\end{example}

\begin{theorem} \label{thm.types}
Let $G$ be a connected simple graph which is not 
isomorphic to either ${\mathcal C}_n^+$, $n \geq 3$, or to the cone $C(A_3)$, and does not contain an induced subgraph
$H$ isomorphic to either ${\mathcal C}_n^+$, $n \geq 3$, or  $C(A_3)$.  
Then $G$ is either a tree, or a cycle  ${\mathcal C}_n$, $n \geq 3$, or 
a complete graph ${\mathcal K}_n$, $n \geq 4$, or a  complete bipartite graph ${\mathcal K}(p,q)$, $p,q \geq 2$.
\end{theorem}
\begin{proof}
Assume that $G$ is neither a tree nor a cycle.
Let $m$ denote the length of the shortest cycle in $G$. Since we assume that $G$ is simple and not a tree, $m \geq 3$. Consider such a cycle ${\mathcal C}$ of length $m$ in $G$,
with consecutive vertices $w_1,\dots, w_m$. For $i=1,\dots, m-1$, the vertex $w_i$ is linked to the vertex $w_{i+1}$ by exactly one edge, and $w_m$ is linked to $w_1$ be one edge. 
This cycle has to be an induced subgraph of $G$. Indeed, if there existed an edge between two vertices of  ${\mathcal C}$ that are not consecutive, the graph $G$ would have a cycle of length smaller than $m$.  We show below that the case $m\geq 5$ is impossible, that when $m=4$, the graph $G$ is of the form ${\mathcal K}(p,q)$, and that when $m=3$, the graph $G$ is of the form ${\mathcal K}_n$. 

Since $G \neq {\mathcal C}$,   let $w$ be a vertex of $G$ not contained in ${\mathcal C}$, but connected by an edge to a vertex of ${\mathcal C}$.
Without loss of generality, we can assume that $w$ is connected to $w_1$. If $w$ is not connected to any other vertices of ${\mathcal C}$, then 
$G$ is equal to ${\mathcal C}_m^+$, or contains ${\mathcal C}_m^+$ as induced subgraph, contradicting our hypothesis. Let us then assume that $w$ is also connected by an edge to a vertex $w_i$
with $i>1$. This is not possible if $m \geq 5$. Indeed, if $m \geq 5$, then $G$ would contain a cycle of length smaller than $m$. 

Assume now that $m=4$, with the cycle ${\mathcal C}$ having vertices $w_1,w_2,w_3, w_4$, and with a vertex $w$ of $G$ not on ${\mathcal C}$ connected to $w_1$. In order for $G$ not to have cycles of length $3$, the vertex $w$ can only be connected to $w_3$. Let $w'$ be any other vertex of $G$ connected to a vertex of ${\mathcal C}$. We claim that $w'$ is then only connected to both $w_1$ and $w_3$. Indeed, assume that $w'$ is connected to $w_2$. 
Then it must be connected to $w_4$, otherwise $G$ contains a cycle of length $3$ or has ${\mathcal C}_4^+$ as induced subgraph. But now we again find a contradiction by considering the cycle $\{w_2,w',w_4,w_3\}$
with $w$ attached to $w_3$. This is an induced subgraph isomorphic to ${\mathcal C}_4^+$, contradicting our hypothesis. 
We have shown so far that $G$ contains a bipartite graph ${\mathcal K}(2,q)$ for some $q \geq 3$, with 
$\{w_1, w_3\}$ being the first set of $2$ vertices in the partition, and $\{w_2,w_4,w,w',\dots\}$ the second set of $q\geq 3 $ vertices. 

Consider now a maximal subgraph of $G$ of the form ${\mathcal K}(r,s)$, with $r \geq 2$ and $s \geq q$. 
By maximal we mean that $G$ does not contain a subgraph of the form ${\mathcal K}(r+1,s)$ or ${\mathcal K}(r,s+1)$. We claim then that $G={\mathcal K}(r,s)$.
For convenience, let us denote by $\{u_1,\dots, u_r\}$ and $\{t_1,\dots, t_s\}$ the vertices of the bipartite graph ${\mathcal K}(r,s)$, so that in  ${\mathcal K}(r,s)$, there are no edges between vertices in $\{u_1,\dots, u_r\}$ and no edges between vertices in  $\{t_1,\dots, t_s\}$.

Suppose that $G\neq {\mathcal K}(r,s)$. There cannot exist an edge of $G$ that links two vertices of ${\mathcal K}(r,s)$ that is not already an edge of ${\mathcal K}(r,s)$, since otherwise the graph $G$ would contain a cycle of length $3$. Thus there exists  a vertex $v$ of $G$ that is not a vertex of 
${\mathcal K}(r,s)$, and is linked by at least one edge to a vertex of ${\mathcal K}(r,s)$. Without loss of generality, we can assume that $v$ is linked to $u_1$. 
If $v$ is not linked to any other vertex of ${\mathcal K}(r,s)$, then $G$ contains an induced subgraph of the form ${\mathcal C}_4^+$, which is a contradiction. If $v$ is linked to any of the vertices $\{t_1,\dots, t_s\}$, 
then $G$ contains a cycle of length $3$, again a contradiction. Suppose now that 
$v$ is linked to $u_2$, but that for some $i \leq r$, $v$ is not linked to $u_i$. 
Then $\{v,u_1,t_1,u_2\}$ are the vertices of a $4$-cycle, and adding $u_i$ to it gives an induced subgraph of $G$ of the form  ${\mathcal C}_4^+$, again a contradiction. 
Thus we find that $G$ contains a graph of the form ${\mathcal K}(r,s+1)$, and this is not possible by maximality of the graph  ${\mathcal K}(r,s)$. Therefore,
$G= {\mathcal K}(r,s)$. 

Let us consider now the case $m=3$, with the cycle ${\mathcal C}$ having vertices $w_1,w_2,w_3$, and with a vertex $w$ of $G$ not on ${\mathcal C}$ connected to $w_1$. Since $G$ is not equal to ${\mathcal C}_3^+$, and does not contain ${\mathcal C}_3^+$ as induced subgraph, $w$ is connected to a second vertex of ${\mathcal C}$, say, without loss of generality, $w_2$. If $w$ is not connected to $w_3$, then $G$ contains the cone $C(A_3)$ as induced subgraph, contradicting our hypothesis. Hence, $w$ is also connected to $w_3$, and so $G$ contains ${\mathcal K}_4$ as induced subgraph.

Consider now a subgraph of $G$ that is of the form ${\mathcal K}_r$ for some $r \geq 4$ and which is maximal, in the sense that $G$ does not contain a subgraph isomorphic to  ${\mathcal K}_{r+1}$. We claim then that $G={\mathcal K}_r$.

Assume that $G\neq {\mathcal K}_r$. Since $G$ is simple, there must then exist a vertex $v$ of $G$ that is not a vertex of ${\mathcal K}_r$. For convenience, let us denote by $\{u_1,\dots, u_r\}$ the vertices of ${\mathcal K}_r$, and assume that $v $ is linked to $u_1$. If $v$ is not linked to any other vertex of ${\mathcal K}_r$, then $G$ contains an induced subgraph of the form ${\mathcal C}_3^+$, which is a contradiction. Suppose then that $v$ is linked to $u_2$, and that there exists some $u_i$ which is not linked to $v$. Then $\{v,u_1,u_2,u_i\}$ are the vertices of an induced subgraph of $G$ of the form $C(A_3)$, again a contradiction. Suppose then that $v$ is linked in $G$ to all vertices of ${\mathcal K}_r$. This is impossible since $G$ would then contain a subgraph of the form ${\mathcal K}_{r+1}$, contradicting the maximality of ${\mathcal K}_r$. 
Hence, $G={\mathcal K}_r$.
\end{proof}

\begin{emp} \label{proof.{thm.main2}}
{\it Proof of Theorem \ref{thm.main2}.}
Let $G$ be a connected simple graph that is neither a tree, a cycle, a complete bipartite graph ${\mathcal K}(p,q)$, a complete graph ${\mathcal K}_n$, nor 
 ${\mathcal C}_n^+$, $n \geq 3$, or the cone $C(A_3)$. 
Theorem \ref{thm.types} shows then that
$G$ has to contain an induced subgraph   $H$ of the form  $C(A_3)$ or ${\mathcal C}_n^+$ for some $n \geq 3$. 
All these graphs $H$ are such that $1 \in V_H(2)$ (see Proposition \ref{cor.tadpole} and Example \ref{ex.cone}). We can thus apply Theorem \ref{thm.main1} (b) to obtain that $V_G(2)$ contains ${\mathbb Z}_{>0}$
for such $G$.
 
When $n \geq 14$, the graph  ${\mathcal K}_n$ has the property that $V_{{\mathcal K}_n}(2)$ contains ${\mathbb Z}_{>0}$ (see Corollary \ref{cor.complete} (b)).
When $n \geq 8$, the graph  ${\mathcal C}_n^+$ has the property that $V_{{\mathcal C}_n^+}(2)$ contains ${\mathbb Z}_{>0}$ (see Proposition \ref{cor.tadpole}).  \qed
\end{emp}

Let ${\mathcal K}_n^+$ denote the graph obtained
from the complete graph ${\mathcal K}_n$ on $n$ vertices by adding one vertex and linking it to  ${\mathcal K}_n$ by exactly one edge. Such graph is a type of {\it lollipop} graph.  Recall that the {\it wheel} ${\mathcal W}_n$ is the cone on the cycle ${\mathcal C}_n$. 
In particular, ${\mathcal W}_3$ is the complete graph $ {\mathcal K}_4$ on $4$ vertices.

\begin{corollary} Let $n \geq 4$. We have $V_{G}(2) = {\mathbb Z}_{\geq 0}$ when $G$ is one of the following graphs:
\begin{enumerate}[\rm (a)]
\item  $G={\mathcal K}_n^+$.
\item   $G$ is  the cone $C(A_n)$ on the path $A_n$ on $n$ vertices. 
\item  $G$ is   the wheel ${\mathcal W}_n$.
\item $G$ is the graph on $n \geq 5$ vertices obtained from the cycle 
${\mathcal C}_n$ 
by adding 
 a new edge linking two vertices which are not already connected in the cycle.
\end{enumerate}
\end{corollary}
\begin{proof} We can apply Theorem \ref{thm.main2} to each graph in the statement of the corollary to obtain that $V_{G}(2) \supset {\mathbb Z}_{> 0}$. In each case, we can further strengthen this result by showing that $0 \in V_G(2)$ as follows.

(a) Since ${\mathcal K}_n$ is the cone on ${\mathcal K}_{n-1}$, 
 Theorem \ref{thm.main1} (b) (ii) can be used to show that $0 \in V_G(2)$.

\if false
Recall that ${\mathcal K}_3^+= {\mathcal C}_3^+$.
Example \ref{ex.lollipop} shows that when $n=3$,  $1 \in V_{{\mathcal C}_3^+}(2)$. Since ${\mathcal K}_n$ is the cone on ${\mathcal K}_{n-1}$, 
we can use induction on ${\mathcal K}_{n-1}^+$ and apply Theorem \ref{thm.main1} (b)  to 
obtain that $V_{{\mathcal K}_n^+}(2) = {\mathbb Z}_{\geq 0}$.
\fi

(b) and (d)
\if false
 The cone $G=C(A_n)$ contains the graph $H={\mathcal C}_3^+$ as induced subgraph as soon as $n \geq 4$. Example \ref{ex.lollipop} shows that $1 \in V_{{\mathcal C}_3^+}(2)$.
It follows then immediately  from Theorem \ref{thm.main1} (b) that $V_{G}(2) \supseteq {\mathbb Z}_{>0}$.
The same result follows from the fact that $C(A_n)$ contains the graph $H=C(A_3)$ as induced subgraph as soon as $n \geq 4$. Example \ref{ex.cone} shows that $1 \in V_{C(A_3)}(2)$. 
\fi
To show that $0 \in V_{G}(2)$, we take the usual Laplacian of $G$, and note that its associated critical group is always cyclic (\cite{Lor08}, Corollary 6.7).

(c) The wheel ${\mathcal W}_n$ is the cone on the cycle $ {\mathcal C}_n$. 
By removing a vertex $v$ on the wheel that belongs to the original cycle, we obtain an induced subgraph isomorphic to the cone $C(A_{n-1})$. We have shown above that $1 \in V_{C(A_{n-1})}(2)$ when $n-1 \geq 3$. 
It follows then immediately  from Theorem \ref{thm.main1} (b) (ii) that $V_{G}(2) ={\mathbb Z}_{\geq 0}$. 
\if false
Since the vertex $v$ has degree $3$, we can also use Theorem \ref{thm.main1} (b) to show that $0 \in V_G(2)$. 
\fi 
\if false
Because $n \geq 5$, the graph $G$ has a vertex $v$ of degree $2$ which is connected to another vertex of degree $2$ and to one vertex of degree   $3$. The subgraph $G_v$ consists in a $m$-cycle for some $m \leq n-2$ with a path attached to it. This graph is found to have $1 \in V_{G_v}(2)$ using  Theorem \ref{thm.main1} (b) and Proposition \ref{cor.tadpole}. Thus  Theorem \ref{thm.main1} (b) be applied again to show that $V_{G}(2)\supseteq {\mathbb Z}_{> 0}$.

To show that $0 \in V_{G}(2)$, we take the usual Laplacian of $G$, and note that its associated critical group is always cyclic (\cite{Lor08}, Corollary 6.7). 
\fi
\end{proof}

\if false
\begin{corollary} \label{cor.completeextended}
Let $G={\mathcal K}_n^+$. When $n \geq 4$, $V_{G}(2) = {\mathbb Z}_{\geq 0}$.
\end{corollary}
\begin{proof}
Recall that ${\mathcal K}_3^+= {\mathcal C}_3^+$.
Example \ref{ex.lollipop} shows that when $n=3$,  $1 \in V_{{\mathcal C}_3^+}(2)$.
Since ${\mathcal K}_n$ is the cone on ${\mathcal K}_{n-1}$, 
we can use induction on ${\mathcal K}_{n-1}^+$ and apply Theorem \ref{thm.main1} (b)  to 
obtain that $V_{{\mathcal K}_n^+}(2) = {\mathbb Z}_{\geq 0}$.
\end{proof}

\begin{proposition} \label{pro.addedge}
Let $G$ denote the graph on $n \geq 5$ vertices obtained from the cycle on $n$ vertices
by adding one diagonal edge. 
 Then $V_{G}(2)={\mathbb Z}_{\geq 0}$.
\end{proposition}

\begin{proof}
Because $n \geq 5$, the graph $G$ has a vertex $v$ of degree $2$ which is connected to another vertex of degree $2$ and to one vertex of degree   $3$. The subgraph $G_v$ consists in a $m$-cycle for some $m \leq n-2$ with a path attached to it. This graph is found to have $1 \in V_{G_v}(2)$ using  Theorem \ref{thm.main1} (b) and Proposition \ref{cor.tadpole}. Thus  Theorem \ref{thm.main1} (b) be applied again to show that $V_{G}(2)\supseteq {\mathbb Z}_{> 0}$.

It remains to show that $0 \in V_{G}(2)$. For this, we take the usual Laplacian of $G$, and note that its associated critical group is always cyclic (\cite{Lor08}, Corollary 6.7).
\end{proof}

\begin{corollary}  \label{cor.cone-wheel} 
Let $n \geq 4$. 
\begin{enumerate}[\rm (a)]
\item Let $G$ be  the cone $C(A_n)$ on the path $A_n$ on $n$ vertices. Then $V_{G}(2) = {\mathbb Z}_{\geq 0}$.
\item Let $G$ be   the wheel ${\mathcal W}_n$. Then $V_{G}(2)  = {\mathbb Z}_{\geq 0}$. 
\end{enumerate}
\end{corollary}
\begin{proof} 
(a) The cone $G=C(A_n)$ contains the graph $H={\mathcal C}_3^+$ as induced subgraph as soon as $n \geq 4$. Example \ref{ex.lollipop} shows that $1 \in V_{{\mathcal C}_3^+}(2)$.
It follows then immediately  from Theorem \ref{thm.main1} (b) that $V_{G}(2) \supseteq {\mathbb Z}_{>0}$.
The same result follows from the fact that $C(A_n)$ contains the graph $H=C(A_3)$ as induced subgraph as soon as $n \geq 4$. Example \ref{ex.cone} shows that $1 \in V_{C(A_3)}(2)$. It remains to show that $0 \in V_{G}(2)$. For this, we take the usual Laplacian of $G$, and note that its associated critical group is always cyclic (\cite{Lor08}, Corollary 6.7).

(b) The wheel ${\mathcal W}_n$ is the cone on the cycle $ {\mathcal C}_n$. 
By removing a vertex $v$ on the wheel that belongs to the original cycle, we obtain an induced subgraph isomorphic to the cone $C(A_{n-1})$. We have shown above that $1 \in V_{C(A_{n-1})}(2)$ when $n-1 \geq 3$. 
It follows then immediately  from Theorem \ref{thm.main1} (b) that $V_{G}(2) \supseteq {\mathbb Z}_{>0}$. 
Since the vertex $v$ has degree $3$, we can also use Theorem \ref{thm.main1} (b) to show that $0 \in V_G(2)$. 
\end{proof}
\fi
\begin{example}  Let $G={\mathcal W}_n$.
The critical group $\Phi$ associated to the Laplacian of  ${\mathcal W}_n$ is not cyclic (see \cite{Big}, 9.2).
Starting with the extended cycle 
${\mathcal C}_3^+$,
Theorem \ref{thm.main1} (b) lets us construct on $G$
a different arithmetical structure $(M,R)$ whose associated group $\Phi$ is trivial.

 We note below on the example of ${\mathcal W}_6$ that the coefficients of the diagonal of $M$ quickly become very large with this construction. In the wheel ${\mathcal W}_6$  on the left below, the vertices of the original ${\mathcal C}_3^+$
are indicated in white. We then constructed ${\mathcal W}_6$ by adding in sequence three vertices, whose corresponding coefficients on the diagonal are $46$, $1478$, and $1548583$.
$$
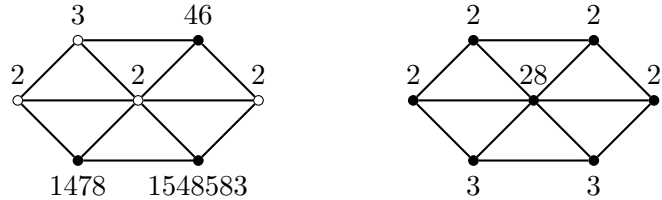
\begin{array}{cc}
\begin{tikzpicture}
[node distance=0.8cm, font=\small] 
\tikzstyle{vertex}=[circle, draw, fill, inner sep=0mm, minimum size=0.7ex]
\node[vertex, fill=none]	(v1)  	at (0,0) 	[label=above:{$2$}] 		{};
\node[]	(dummy1)		[right of=v1, label=above:{}]	{};
\node[vertex, fill=none]	(v2)		[right of=dummy1, label=above:{$2$}]	{};
\node[]	(dummy2)		[right of=v2, label=above:{}]	{};
\node[vertex, fill=none]	(v3)			[right of=dummy2, label=above:{$2$}]	{};
\node[vertex, fill=none]	(u1)			[above of=dummy1, label=above:{$3$}]	{};
\node[vertex]	(w1)			[below of=dummy1,label=below:{$1478$}]	{};
\node[vertex]	(u2)			[above of=dummy2, label=above:{$46$}]	{};
\node[vertex]	(w2)			[below of=dummy2, label=below:{$1548583$}]	{};
 
\draw [thick] (v1)--(v2)--(v3);
\draw [thick] (u1)--(u2)--(v3)--(w2)--(w1)--(v1)--(u1);
\draw[thick] (w2)--(v2)--(u1);
\draw[thick] (w1)--(v2)--(u2);
 
\end{tikzpicture}
& \quad \quad \quad 
\begin{tikzpicture}
[node distance=0.8cm, font=\small] 
\tikzstyle{vertex}=[circle, draw, fill, inner sep=0mm, minimum size=0.7ex]
\node[vertex ]	(v1)  	at (0,0) 	[label=above:{$2$}] 		{};
\node[]	(dummy1)		[right of=v1, label=above:{}]	{};
\node[vertex ]	(v2)		[right of=dummy1, label=above:{$28$}]	{};
\node[]	(dummy2)		[right of=v2, label=above:{}]	{};
\node[vertex ]	(v3)			[right of=dummy2, label=above:{$2$}]	{};
\node[vertex ]	(u1)			[above of=dummy1, label=above:{$2$}]	{};
\node[vertex]	(w1)			[below of=dummy1,label=below:{$3$}]	{};
\node[vertex]	(u2)			[above of=dummy2, label=above:{$2$}]	{};
\node[vertex]	(w2)			[below of=dummy2, label=below:{$3$}]	{};
 
\draw [thick] (v1)--(v2)--(v3);
\draw [thick] (u1)--(u2)--(v3)--(w2)--(w1)--(v1)--(u1);
\draw[thick] (w2)--(v2)--(u1);
\draw[thick] (w1)--(v2)--(u2);
 
\end{tikzpicture}
\end{array}
$$
  
We describe in Lemma \ref{lem.wheel} a different arithmetical structure 
on ${\mathcal W}_{2k}$, where  $\Phi$ is cyclic of order $6k-1$, and where the coefficients on the diagonal do not grow as fast. We have illustrated the case $k=3$ on the right above.

Arithmetical structures on the wheel  ${\mathcal W}_6$  are plentiful. For instance, looking only among structures with $M_G(a_1,\dots, a_7)$ having $a_i \in [2,30]$ for all $i=1,\dots, 7$, we found structures with $198$ different group orders for their associated group $\Phi$. Among the orders found, $94$ are squarefree, so that the corresponding groups are cyclic.
\end{example}

\begin{lemma} \label{lem.wheel}
Order the vertices of ${\mathcal W}_{2k}$ as follows. 
The first vertex in our enumeration is $u$, the unique vertex of degree $2k$. Then we let $v_1,\dots,v_k, w_k,w_{k-1},\dots,w_1$ denote the consecutive vertices on the cycle, all of degree $3$.
The transpose of the vector $R$ is $(1,r_1,\dots,r_k,s_k,\dots, s_1)$,
with $r_i=s_i$ for all $i=1,\dots,k$. We set $s_k=k$, $s_{k-1}=k+(k-1)$, 
and so on, until we get to $s_1=k+\dots+2+1=k(k+1)/2$. 
Note that $a:=2\sum_{i=1}^k s_i= 2\sum_{i=1}^k i^2= k(k+1)(2k+1)/3$.  
 The diagonal of the matrix $M$ of the structure is $(a,2,\dots,2,3,3,2,\dots,2)$. Then $(M,R)$ is an arithmetical structure on ${\mathcal W}_{2k}$
 whose associated group $\Phi$ is cyclic of order $6k-1$.
 \end{lemma}
  
\begin{proof} It is easy to check  that $MR=0$.
 The computation of $\Phi$ can be done as follows. First, since the first coefficient of $R$ is $1$, we can find an integer linear combination of the columns to add to the first column so that the resulting first column is the zero-column. The same operations on the lines produces a new matrix whose first line is the zero-line and whose first column is the zero-column. 
 Let $M'$ denote the bottom right $2k \times 2k$ square submatrix of this matrix. 
 The group $\Phi$ is obtained by doing a row and column reduction
 of the matrix $M'$. In particular, the determinant of this matrix 
 gives us $|\Phi|$. It follows that
 $|\Phi|=d_{{\mathcal C}_{2k}}(2\dots,2,3,3,2, \dots, 2)=6k-1$. Choose now a row and column of $M'$ where the diagonal element is $3$. To show that the group is cyclic, we consider the submatrix $M''$ of $M'$ obtained by removing the chosen row and column. The matrix $M''$ has determinant $4k-1$. Since $6k-1$ and $4k-1$ are always coprime, we find that $\Phi$ is cyclic.  We leave the details to the reader.
\end{proof}

\begin{remark}  
We note here  an arithmetical structure 
 on ${\mathcal W}_{2k+1}$ very similar to the structure described in Lemma \ref{lem.wheel}, with the exception that this new structure  has its group   $\Phi$ isomorphic to $({\mathbb Z}/(2k+1){\mathbb Z})^2$.

 Order the vertices of ${\mathcal W}_{2k+1}$ as follows. 
As before, the first vertex in our enumeration is $u$, the unique vertex of degree $2k+1$. Then we let $v_1,\dots,v_k, u', w_k,w_{k-1},\dots,w_1$ denote the consecutive vertices on the cycle, all of degree $3$.
The transpose of the vector $R$ is $(1,r_1,\dots,r_k,1,s_k,\dots, s_1)$,
with $r_i=s_i$ for all $i=1,\dots,k$. We set $s_k=1+k$, $s_{k-1}=1+k+(k-1)$, 
and so on, until we get to $s_1=1+k+\dots+2+1=1+ k(k+1)/2$. 
Note that $a:=1+2\sum_{i=1}^k s_i= 1+2k+2\sum_{i=1}^k i^2=1+2k+ k(k+1)(2k+1)/3$.  
 The diagonal of the matrix $M$ of the structure is $(a,2,\dots,2,2k+3,2,\dots,2)$.  It is easy to check  that $MR=0$. We can obtain in a similar way that
 $|\Phi|=d_{{\mathcal C}_{2k+1}}(2\dots,2,2k+3,2, \dots, 2)=(2k+1)^2$.
\end{remark}

\if false
\begin{example}
Computations suggest that when $n=4,5,6$, the wheel $G={\mathcal W}_n$ is such that
$V_{d_G}(2)={\mathbb Z}_{\geq 0}$. 

Note that removing one vertex of $ {\mathcal W}_n$ 
does not produce in general a subgraph $H$ where it is known that $1 \in V_H(2)$. For instance,
when $n=4$, removing one vertex of $ {\mathcal W}_4$ 
gives either the cycle on $4$ vertices, or the cone on the path $A_3$. 
Neither of these graphs $H$ has $1 \in V_{d_H}(2)$. 
\end{example}
\fi
\end{section}

\begin{section}{Third Main Theorem} 

We first recall the following definitions needed for the proof of 
Theorem  \ref{thm.density}, the main theorem of this section. 

\begin{emp}  \label{rem.naturaldensity} \label{def.dense}
Let $S \subset {\mathbb Z}_{\geq 0}$ be any subset. Let ${\displaystyle S(\ell):=\{0,1,2,\ldots ,\ell\}\cap S}$ and ${\displaystyle s(\ell):=|S(\ell)|}$.
Recall that the {\it lower density }  $ \underline{d}(S)$   of $S$  is defined as $$ \underline{d}(S) := \liminf_{\ell \rightarrow \infty} \frac{ s(\ell) }{\ell}.$$ 
Similarly, the upper density $\overline{d}(S) $ of $S$ is given by
$ \overline{d}(S) := \limsup_{\ell \rightarrow \infty} \frac{s(\ell)}{\ell} $.
If both $ \underline{d}(S)$ and $\overline{d}(S) $ exist and are equal, the {\it natural density} $d(S)$ of $S$ is defined as $d(S):=  \underline{d}(S)$.

We say that $S$ is {\it dense in ${\mathbb Z}_{\geq 0}$} if $d(S)=1$. 
When $S$ is dense, its complement in ${\mathbb Z}_{\geq 0}$ has density $0$. 
Any finite set has density $0$. Let $a, b \in {\mathbb Z}_{\geq 0}$, $a \neq 0$. A set of the form 
$\{ am+b \mid m \geq 0\}$  is called an {\it arithmetic progression} and has density $1/a$.

For later use, we note here the following facts. 
Consider a set $S$ of positive integers which contains a union 
$$U:=\bigcup_{i=1}^k \left(\bigcup_{j=1}^{r_i} \{ a_it+b_{ij} \mid t \geq 0\}\right)$$ of arithmetic progressions.  
Then the lower density $\underline{d}(S)$ of $S$ satisfies  $\underline{d}(S) \geq d(U)$. 
When the $a_i$ are pairwise coprime, and for each $i$, the $r_i$ arithmetic progressions are distinct, we find that 
$$d(U)=  1- \prod_{i=1}^k \left(1-\frac{r_i}{a_i}\right).$$
\end{emp}

\begin{theorem} \label{thm.density} Let $G$ be a connected graph.
 Let $v \in G$ be a vertex such that the subgraph $G_v$ has one of the following properties:
\begin{enumerate}[\rm (a)]
\item The complement of $V_{d_{G_v}}(2)$  in ${\mathbb Z}_{\geq 0}$
is finite. 
\item Up to possibly reordering the vertices of $G_v$,  there exist $a_2,\dots,a_{n-1} \in {\mathbb Z}_{\geq 2}$ such that
$d_{G_v}(t,a_2,\dots, a_{n-1})=\alpha t -\beta$ with $\alpha \in {\mathbb Z}_{>0}$, $\beta \in {\mathbb Z}$, and $\gcd(\alpha, \beta)=1$. 
\item  $V_{d_{G_v}}(2)$ 
 contains an infinite subset of pairwise coprime values $\{u_1,u_2,\dots\}$ 
such that $\lim_{j \to \infty} \prod_{i=1}^j (1-1/u_i)=0$.
\end{enumerate}
Then $V_{d_G}(2)$ is dense in ${\mathbb Z}_{\geq 0}$.
\end{theorem}



\begin{proof}
{\it We claim that Assumption (a) implies 
Assumption (c).} Indeed, since the complement of $V_{d_{G_v}}(2)$ is finite, it must contain all but finitely many prime numbers. Listing the prime numbers in the complement as $\{p_i\}_{i=1}^\infty$, we find that $\lim_{j \to \infty} \sum_{i=1}^j 1/p_i $ diverges, since by Euler's theorem the sum of the reciprocals of all primes diverges (\cite{I-R}, page 21), and our set of primes misses only finitely many primes by hypothesis.
Lemma 5.1.1 in \cite{AFG} can be used to deduce that $\lim_{j \to \infty} \prod_{i=1}^j (1-1/p_i)=0 $, as desired. 

\smallskip
{\it We claim that Assumption (b) also implies Assumption (c).} Indeed, our assumption that $\gcd(\alpha, \beta)=1$ allows us to use Dirichlet's Theorem on primes in arithmetic progression: 
The set ${\mathcal P}$ of primes in the arithmetic progression
$\{ \alpha t-\beta \mid t \in {\mathbb Z}_{\geq 0}\}$ is infinite because the Dirichlet density of the primes in ${\mathcal P}$  is equal to $1/\alpha$. This in turn implies that $\sum_{p \in {\mathcal P}} 1/p$ is infinite (see \cite{I-R}, page 251). As before, we conclude using 
Lemma 5.1.1 in \cite{AFG} that  $\lim_{p \in {\mathcal P}} \prod_p (1-1/p)=0 $, as desired.

\smallskip
Let us now prove the theorem when Assumption (c) holds. Without loss of generality, we can assume that $v=v_1$. 
For each value $u_i$, choose $a_{i,2},\dots, a_{i,n} \geq 2$ such that $d_{G_v}(a_{i,2},\dots, a_{i,n})=u_i$. 
Consider the polynomial $d_G(t,a_{i,2},\dots, a_{i,n})$.
By hypothesis, the matrix $M_G(t,a_{i,2},\dots, a_{i,n} )$ has a lower right $n-1 \times n-1$ 
submatrix $ M_{G_v}(a_{i,2},\dots, a_{i,n})$  
of determinant $u_i$.
Thus $d_G(t,a_{i,2},\dots, a_{i,n} ) = u_it-w_i$ for some $w_i \in {\mathbb Z}$. 

Let $U_i$ denote the set of positive values taken by $ u_it-w_i$ when $t \geq 2$ if $w_i \leq 0$, and when 
$t \geq 2+ w_i/u_i$ if $w_i \geq 0$. It is clear that, since up to finitely many values the set
$U_i$ is an arithmetic progression, the density of $U_i$ is $1/u_i$. By construction, $U_i \subseteq V_{d_G}(2)$. 
  Since the elements of the set $\{u_1,u_2,\dots\}$ are pairwise coprime,
we find that 
the  union $\cup_{i=1}^j U_i$ has lower density  equal to $1-\prod_{i=1}^j (1-1/u_i)$. 
Clearly, $V_{d_G}(2)$ contains $\cup_{i=1}^\infty U_i$,
and taking now the limiting value as $j \to \infty$, we find that $\underline{d}(V_{d_G}(2)) \geq 1$. 
It follows that $d(V_{d_G}(2)) = 1$ and $V_{d_G}(2)$ is dense in ${\mathbb Z}_{\geq 0}$.
\end{proof}

In our next corollary, the symbols $A_n$ and $D_n$ denote Dynkin diagrams, whose definition is  recalled in \ref{def.Dynkin}.
\begin{corollary} \label{cor.density}
Suppose that a graph $G$ on $n+1$ vertices contains an induced subgraph $H$ on $n$ vertices  of the form $G_v$ such that
$H$ is either $A_n$, ${\mathcal C}_n$,  $D_n$, ${\mathcal S}_n$, ${\mathcal K}_n$, or ${\mathcal K}(2,n-2)$.
Then $V_{d_G}(2)$ is dense in ${\mathbb Z}_{\geq 0}$.
\end{corollary}
\begin{proof}
We exhibit for each graph $H$ a choice of diagonal elements $(t,a_2,\dots,a_n)$ with $a_i \geq 2$ such that 
$d_H(t,a_2,\dots,a_n)=\alpha t-\beta$ with $\gcd(\alpha, \beta)=1$.
We can then use Theorem \ref{thm.density} (b) to conclude that $V_{d_G}(2)$ is dense in ${\mathbb Z}_{\geq 0}$.

\smallskip
(i) $H=A_n$. Assume that $v_1$ is a vertex of degree $1$. 
We find that $$d_H(t,2,\dots,2)=nt-(n-1).$$ Clearly $\gcd(n,n-1)=1$.

\smallskip
(ii) $H={\mathcal C}_n$.  Label the vertices consecutively along the cycle. 
Recall that $n$ is the number of spanning trees of $H$, and it follows from 
the proof of Proposition \ref{pro.arithmeticalstructure} (a) that $d_H(t,2,\dots,2,2)=n(t-2)$.
We find that 
$$d_H(t,2,\dots,2,3)=n(t-2)+ (n-1)t-(n-2)= (2n-1)t-(3n-2).$$
Again, we find that $\gcd(2n-1, 3n-2)=1$.

\smallskip
(iii) $H=D_n$. This star-shaped graph has a single vertex of degree $3$, and three terminal chains, each ending with a vertex of degree $1$. Assume that $v_1$ is the vertex of degree $1$ on the chain of length $n-3$ in $H$, and that $v_n$ is another vertex of degree $1$. 
We find that $$d_H(t,2,\dots,2,2,3)=(n+3)t-(n+2),$$ and clearly $\gcd(n+3,n+2)=1$. To compute the coefficient of $t$ in this expression, 
note that it equals $d_{H_{v_1}}(2,\dots,2,3)$, which is $4+(n-1)$, with $4=d_{D_{n-1}}(2,\dots,2)$ and $n-1=d_{A_{n-2}}(2,\dots,2)$. The constant coefficient is similarly obtained as $4+(n-2)$.

\smallskip
(iv) $H={\mathcal S}_n$.  Assume that the vertex $v_1$ has degree $n-1$. We     use $(t,a_2,\dots,a_n)$ with $a_2=2$ and $a_{i+1}=i$-th prime number. We find that
 $$d_H(t,a_2,\dots,a_n)= (a_2 \cdots a_n)t-(a_2 \cdots a_n)(\frac{1}{a_2}+\dots + \frac{1}{a_n}) .$$
Since $a_2, \dots,a_n$ are distinct primes, we find that
 $ (a_2 \cdots a_n)(\frac{1}{a_2}+\dots + \frac{1}{a_n})$
  is coprime to $(a_2 \cdots a_n)$.

\smallskip
(v) $H={\mathcal K}_n$. Set $a_i+1=b_i$ and choose 
 $a_2+1=3$, and $a_{i}+1=i$-th prime number. 
We find using \eqref{det} that
$$
\begin{array}{rl}
d_H(t,a_2,\dots,a_n)= &
(t+1)(b_2 \cdots b_n)- (t+1)(b_2 \cdots b_n)(\sum_{i=2}^n \frac{1}{b_i})- (b_2 \cdots b_n)\\
= & t(b_2 \cdots b_n)(1-\sum_{i=2}^n \frac{1}{b_i})-(b_2 \cdots b_n)(\sum_{i=2}^n \frac{1}{b_i}).
\end{array}
$$
 Since $b_2, \dots,b_n$ are distinct primes,
 $(b_2 \cdots b_n)$ is coprime to $ (b_2 \cdots b_n)(\frac{1}{b_2}+\dots + \frac{1}{b_n})$.
 
 \smallskip
(vi) $H={\mathcal K}(2,n-2)$. We  compute the determinant of 
$$
M=M_H(t,x,y_1,\dots, y_{n-2})=\left(
\begin{array}{ccccc}
t & 0 & -1& \dots & -1  \\
0 & x& -1& \dots &   -1 \\
-1 & -1 & y_1 & 0  & 0 \\
\vdots & \vdots  &0 & \ddots  & 0 \\
-1 & -1& 0 &   & y_{n-2} \\
\end{array}
\right).
$$ 
Subtracting the second row from the first, and the resulting second column from the first, we obtain
$$
\left(
\begin{array}{ccccc}
t+x & -x & 0& \dots & 0  \\
-x & x& -1& \dots &   -1 \\
0 & -1 & y_1 & 0  & 0 \\
\vdots & \vdots  &0 & \ddots  & 0 \\
0 & -1& 0 &   & y_{n-2} \\
\end{array}
\right).
$$ 
Expanding the determinant along the first row gives
$$
\det(M)= 
\left(xy_1\cdots y_{n-2} - y_1\cdots y_{n-2}(\sum_{i=1}^{n-2} \frac{1}{y_i})\right)t 
-x y_1\cdots y_{n-2}(\sum_{i=1}^{n-2} \frac{1}{y_i}) .
$$
We choose $y_1,\dots, y_{n-2}$ to be the first $n-2$  prime numbers. We choose $x$ to be coprime to 
$y_1\cdots y_{n-2}(\sum_{i=1}^{n-2} \frac{1}{y_i})$.
\end{proof}

\begin{emp} \label{proof.{thm.denselist}}
{\it Proof of Theorem \ref{thm.denselist}.}
 We proceed with a case-by-case verification.
 
(1) Let $G$ be a Dynkin diagram, an extended Dynkin diagram, or the cone $C(A_3)$. Then $G$ contains an induced subgraph of the form $G_v=A_n$ for some $n$. Corollary \ref{cor.density} applies.

(2) Let $G$ be a complete graph ${\mathcal K}_n$, or an extended complete graph ${\mathcal K}_n^+$.
Then $G$ contains an induced subgraph of the form $G_v={\mathcal K}_{n-1}$ (resp.\ ${\mathcal K}_{n}$). Corollary \ref{cor.density} applies.
Let $G$ be a cycle ${\mathcal C}_n$,  an extended cycle ${\mathcal C}_n^+$, or the wheel ${\mathcal W}_n$.
Then $G$ contains an induced subgraph of the form $G_v=A_n$ (resp.\ $A_{n+1}$, resp.\ ${\mathcal C}_n$). Corollary \ref{cor.density} applies.

(3) Let $G$ be a star ${\mathcal S}_n$,   an extended star ${\mathcal S}_n^+$, or the complete bipartite ${\mathcal K}(2,n)$ or ${\mathcal K}(3,n)$. Then $G$ contains an induced subgraph of the form $G_v={\mathcal S}_{n-1}$ (resp.\ ${\mathcal S}_n$, resp.\ ${\mathcal S}_{n+1}$, resp.  ${\mathcal K}(2,n)$). Corollary \ref{cor.density} applies.
 
 \if false
We have $V_{d_G}(2)\supseteq {\mathbb Z}_{> 0}$ when $G$ is
${\mathcal K}_n$, $n\geq 14$ (see Corollary \ref{cor.complete}).

  or $G={\mathcal K}_n^+$, $n\geq 4$ (see Corollary \ref{cor.completeextended}),
   or $G={\mathcal C}_n^+$, $n\geq 8$ (see Proposition \ref{cor.tadpole}),
   or ${\mathcal W}_n$, $n\geq 4$ (see Corollary \ref{cor.cone-wheel} (b)). 
\fi
 
\end{emp}

It is natural to wonder whether a vertex $v$ such that Hypothesis (b) in Theorem \ref{thm.density} holds might exist for all graphs. 
We have not been able to answer this question beyond the following extension result.

\begin{lemma} Suppose that $G$ is a graph on $n$ vertices such that for some choice of 
$a_2,\dots, a_n \geq 2$, we have $d_G(t,a_2,\dots, a_n)=\alpha t -\beta$ with 
$\gcd(\alpha,\beta)=1$. Let $G^+$ denote the graph obtained by attaching a new vertex $v_0$ with $e \geq 1$ edges to the vertex $v_1$. 
Let $q \geq 1$ denote any integer  coprime to $e\alpha\beta$. 
Then $d_{G^+}(q,t,a_2,\dots, a_n)$
is of the form $\alpha't - \beta'$ with $\gcd(\alpha', \beta')=1$. 
\end{lemma}
\begin{proof} An explicit computation shows that
 $d_{G^+}(q,t,a_2,\dots, a_n)=q(\alpha t -\beta)-e^2\alpha=q\alpha t -(q\beta+e^2\alpha)$.
 It follows from our hypotheses that $q\alpha$ is coprime to
 $(q\beta+e^2\alpha)$.
 \end{proof}
 \begin{remark}
The connected simple graphs on four vertices consist of the path
$ A_4$, the star $ D_4$, the two graphs  $ {\mathcal C}_3^+$  and $ {\mathcal C}_4$
with Betti number equal to $1$,  the cone on $A_3$, and the complete graph ${\mathcal K}_4$.
Theorem \ref{thm.denselist} and Corollary \ref{cor.density} show that for such graph $G$, the set $V_{d_G}(2)$ is dense in ${\mathbb Z}_{\geq 0}$.
Computations suggest that except possibly for 
 ${\mathcal K}_4$, the complement of $V_{d_G}(2)$ in ${\mathbb Z}_{\geq 0}$ might be finite. 
 \end{remark} 
 \end{section}
 
\begin{section}{Dynkin diagrams} \label{sec.Dynkin}

\begin{emp} \label{def.Dynkin}
Recall  the following terminology.

\smallskip
$A_n$, $n\geq 2$: the {\it path} (or {\it chain}), on $n$ vertices.
\smallskip



$\tilde{D_n}$, $n\geq 4$: a chain on $n-1$ vertices, with two additional vertices, attached to the two vertices of degree $2$ of the chain that are linked to a vertex of degree $1$ (when $n=4$, there is only one such vertex, and in this case both new vertices are attached to this vertex).
The graph $D_n$, $n\geq 4$, is obtained
from  $\tilde{D_n}$ by removing one of the two additional vertices. Such vertex is indicated in white below.

$$
 \tilde{D_n} \quad
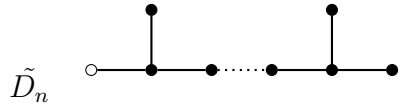
\begin{tikzpicture}
[node distance=0.8cm, font=\small] 
\tikzstyle{vertex}=[circle, draw, fill, inner sep=0mm, minimum size=0.8ex]
\node[vertex, fill=none]	(v1)  	at (0,1) 	[label=below:{}] 		{};
\node[vertex]	(v33)			[right of=v1, label=below:{}]	{};
 \node[vertex]	(w)			[above of=v33, label=above:{}]	{};
\node[vertex]	(v44)			[right of=v33, label=below:{}]	{};
\node[vertex]	(v2)			[right of=v44, label=below:{}]	{};
\node[vertex]	(v3)			[right of=v2, label=below:{}]	{};
\node[vertex]	(v)			[above of=v3,  label=right:{ ${}$}]	{};
\node[vertex]	(v4)			[right of=v3, label=below:{}]	{};
\draw [thick] (v1)--(v33)--(v44);
\draw [thick] (v2)--(v3)--(v4);
\draw [thick] (v)--(v3);
\draw [thick, dotted] (v44)--(v2);
\draw[thick] (v33)--(w);
\end{tikzpicture}
$$

$\tilde{E_n}$, $n=6,7,8$: a tree on $n+1$ vertices described below. Removing one vertex from $\tilde{E_n}$
produces the tree $E_n$ on $n$ vertices. The vertex that needs to be removed is marked in white below.

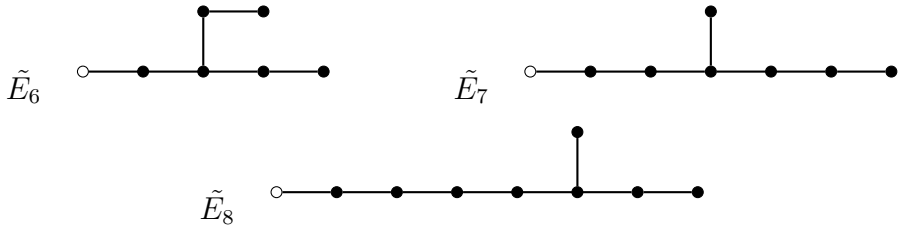
\if false
We draw below the graphs $\tilde{D_n}$, $n\geq 5$, and $\tilde{E_n}$, $n=6,7,8$, marking one vertex in white on each graph. Removing the marked vertex along with the associated edge, we obtain
the graphs $ D_n$, $n\geq 5$, and $E_n$, $n=6,7,8$, respectively.
\fi
 

$$
\begin{array}{cc}
\tilde{E_6} \quad
\begin{tikzpicture}
[node distance=0.8cm, font=\small] 
\tikzstyle{vertex}=[circle, draw, fill, inner sep=0mm, minimum size=0.8ex]
\node[vertex, fill=none]	(v1)  	at (0,0) 	[label=below:{}] 		{};
\node[vertex]	(v33)			[right of=v1, label=below:{}]	{}; 
\node[vertex]	(v44)			[right of=v33, label=below:{}]	{};
\node[vertex]	(v2)			[right of=v44, label=below:{}]	{};
\node[vertex]	(v3)			[right of=v2, label=below:{}]	{};
\node[vertex]	(v)			[above of=v44,  label=right:{ ${}$}]	{};
\node[vertex]	(v5)			[right of=v, label=below:{}]	{};
\draw [thick] (v1)--(v2)--(v3);
\draw [thick] (v)--(v44);
\draw[thick] (v)--(v5);
\end{tikzpicture}
 
& \quad \quad \quad
 
\tilde{E_7} \quad
\begin{tikzpicture}
[node distance=0.8cm, font=\small] 
\tikzstyle{vertex}=[circle, draw, fill, inner sep=0mm, minimum size=0.8ex]
\node[vertex, fill=none]	(v1)  	at (0,0) 	[label=below:{}] 		{};
\node[vertex]	(v33)			[right of=v1, label=below:{}]	{};
 
\node[vertex]	(v44)			[right of=v33, label=below:{}]	{};
\node[vertex]	(v2)			[right of=v44, label=below:{}]	{};
\node[vertex]	(v3)			[right of=v2, label=below:{}]	{};
\node[vertex]	(v)			[above of=v2,  label=right:{ ${}$}]	{};
\node[vertex]	(v4)			[right of=v3, label=below:{}]	{};
\node[vertex]	(v5)			[right of=v4, label=below:{}]	{};
\draw [thick] (v1)--(v2)--(v3)--(v4);
\draw [thick] (v)--(v2);
\draw[thick] (v4)--(v5);
\end{tikzpicture}
\end{array}
$$

$$
\tilde{E_8} \quad
\begin{tikzpicture}
[node distance=0.8cm, font=\small] 
\tikzstyle{vertex}=[circle, draw, fill, inner sep=0mm, minimum size=0.8ex]
\node[vertex, fill=none]	(v1)  	at (0,0) 	[label=below:{}] 		{};
\node[vertex]	(v22)		[right of=v1, label=below:{}]	{};
\node[vertex]	(v33)			[right of=v22, label=below:{}]	{};
 
\node[vertex]	(v44)			[right of=v33, label=below:{}]	{};
\node[vertex]	(v2)			[right of=v44, label=below:{}]	{};
\node[vertex]	(v3)			[right of=v2, label=below:{}]	{};
\node[vertex]	(v)			[above of=v3,  label=right:{ ${}$}]	{};
\node[vertex]	(v4)			[right of=v3, label=below:{}]	{};
\node[vertex]	(v5)			[right of=v4, label=below:{}]	{};
\draw [thick] (v1)--(v22);
\draw [thick] (v22)--(v33)--(v44)--(v4);
\draw [thick] (v)--(v3);
\draw[thick] (v4)--(v5);
\end{tikzpicture}
$$
The graphs $A_n$ $(n \geq 2)$, $D_n$ $(n\geq 4)$, and $E_n$,  $n=6,7,8$, are called {\it Dynkin  diagrams} and have $n$ vertices.
The graphs $\tilde{D_n}$, $n\geq 4$, and $\tilde{E_n}$, $n=6,7,8$, are called {\it extended Dynkin diagrams} or {\it affine Dynkin diagrams}, and have $n+1$ vertices. In the context of elliptic curves, they are called {\it Kodaira types}, and are denoted by 
${\rm I}_n^*$, $n \geq 0$, and ${\rm IV}^*, {\rm III}^*$, and ${\rm II}^*$, respectively (see for instance \cite{Tat}, page 46). The notation ${\rm I}_n$ refers in this context to the cycle ${\mathcal C}_n$.
\end{emp}

The following proposition is 
well-known. 
Part (a) is found in \cite{Dur}, Lemma 3.1, with proof  and a reference to \cite{Hir}, Satz page 219. Part (b) is stated in \cite{Hir}, page 228. 

\begin{proposition} \label{pro.DynkinAll}
Let $G$ be a connected graph on $n$ vertices. 
\begin{enumerate}[\rm (a)]
\item  If $M_G(2,\dots,2)$ is positive definite, then $G$ is either $A_n$ with $n \geq 2$, $D_n$ with $n \geq 4$, or $E_n$, $n=6,7,8$. 
\item  If $M_G(2,\dots,2)$ is positive semi-definite and $\det M_G(2,\dots,2)=0$, then $G$ is either ${\mathcal C}_n$ with $n \geq 2$, $\tilde{D_n}$ with $n \geq 4$, or $\tilde{E_n}$, $n=6,7,8$. 
\end{enumerate}
\end{proposition}

\if false
\begin{proof} We leave it to the reader to check that the matrices $M_G(2,\dots,2)$
of the graphs listed in (a) (resp.\ in (b)) are positive definite (resp.\ positive semi-definite). 

Let $G$ be a graph such that  $M:=M_G(2,\dots,2)$ is positive semi-definite. 
We claim that unless $G={\mathcal C}_2$, the graph $G$ can only have at most one edge between any two vertices. Indeed, suppose that $v_1$ and $v_2$ are linked by $ a\geq 1$ edges. Then $M$ contains the submatrix 
$$ \left(
\begin{array}{cc}
2 & -a \\
-a & 2
\end{array}
\right),
$$
whose determinant is $4-a^2$. Since we need $4-a^2>0$ when $G\neq {\mathcal C}_2$,
we find that $a \leq 1$.

The graph $G$ cannot contain any induced subgraph isomorphic to ${\mathcal C}_3$ unless $G={\mathcal C}_3$.

Suppose now that $G$ contains a vertex of degree at least $3$. 
Then  since this vertex cannot be part of a cycle ${\mathcal C}_3$ in $G$ which is an induced subgraph of $G$, 
 it must be part of a star $D_4$ which is an induced subgraph of $G$.

Let us now consider the case where the vertex has degree at least $4$. Again,  since this vertex cannot be part of a cycle ${\mathcal C}_3$ in $G$ which is an induced subgraph of $G$, 
 it must be part of  an induced subgraph isomorphic to $\tilde{D_4}$. In this case, we must have $G=\tilde{D_4}$.

\if false
We claim now that unless $G=\tilde{D_4}$, the degree of any vertex of $G$ is at most $3$.
Indeed, if $G$ has a vertex of degree at least $4$, then $M$ contains a submatrix
of the form 
$$ \left(
\begin{array}{cccc}
2 & -1 & -1 & -1 \\
-1 & 2 & 0& 0\\
-1 &0 & 2 & 0 \\
-1 &0 &0 & 2 \\
\end{array}
\right),
$$

which has determinant $0$. This is not possible unless $M$ is equal to this matrix, in which case $G=\tilde{D_4}$.
\fi

Suppose that $G$ contains two distinct vertices of degree $3$. Then $G$ must contain a subgraph of the form $\tilde{D_n}$ for some $n \geq 5$, or must contain a cycle ${\mathcal C}_n$ of some length $n \geq 3$. Only the former can happen.

We have proved so far that unless $G$ is of the form 
$\tilde{D_n}$ for some $n \geq 4$ or is a cycle ${\mathcal C}_n$ 
for some $n\geq 2$, the graph $G$ is a tree with a unique vertex of degree $3$, and no vertices of higher degree. The smallest such $G$ is $D_4$, and we now look at all possible ways of extending it by  adding one vertex at the time, and making sure that the resulting extension still has only one vertex of degree $3$. 

The start $D_4$ can be extended to $D_5$. The graph $D_5$ can be extended in two ways, to $D_6$ or to $E_6$.  The graph $D_6$ can be extended in two ways, to $D_7$ or to $E_7$.  The graph $E_6$ can be extended in two ways, to $D_7$ or to $\tilde{E_6}$.  No further extensions of $\tilde{E_6}$ are possible.
The graph $D_7$ can be extended in two ways, to $D_8$ or to $E_8$. 
The graph $E_7$ can be extended in three ways, to $E_8$ or to $\tilde{E_7}$ or to a graph that contains $\tilde{E_6}$. No further extensions are possible in the second case, and the third case cannot occur.
The graph $D_8$ can be extended in two ways, to $D_9$ or  to a graph that contains $\tilde{E_8}$.  No further extensions are possible for $\tilde{E_8}$.
The graph $E_8$ can be extended in three ways, to $\tilde{E_8}$,
or to a graph that contains $\tilde{E_7}$, or to a graph that contains $\tilde{E_6}$. No further extension are possible in the first case, and the latter two cases are not possible.
The graph $D_n$, $n\geq 8$ can be extended in two ways, to $D_{n+1}$ or  to a graph that contains $\tilde{E_8}$. The latter case is not possible when $n>8$.
\end{proof}
\fi

When $G$ is a Dynkin diagram, computations indicate that the complement of $V_{d_G}(2)$ in ${\mathbb Z}_{\geq 0}$ is finite, except when $G=A_2$.
We present below  some data for the Dynkin diagrams $D_n$ and $E_n$.
The data for the chains $A_n$ is presented in Example \ref{ex.paths}.

\begin{proposition} \label{pro.Dynkin}
\begin{enumerate}[\rm (a)]
\item If $G= A_n$, $n \geq 2$, then $V_{d_G}(2) \subseteq {\mathbb Z}_{\geq n+1}$.
\item If $G= D_n$, $n \geq 4$, then $V_{d_G}(2) \subseteq {\mathbb Z}_{\geq 4}$.
\item If $G= E_6$, then $V_{d_G}(2) \subseteq {\mathbb Z}_{\geq 3}$.
Computations indicate that the complement of $V_{d_G}(2)$ in ${\mathbb Z}_{\geq 0}$ is contained in 
$$L_{E_6}:=
\begin{array}{c} 
[ 0, 1, 2, 4, 5, 6, 8, 10, 12, 14, 16, 17, 20, 24, 26, 28, 30, 32, 34, 38, \\  44, 46, 48, 
56, 60, 64, 74, 80, 88, 92, 98, 132, 158, 170 ].
\end{array}
$$
\item If $G= E_7$, then $V_{d_G}(2) \subseteq {\mathbb Z}_{\geq 2}$.
Computations indicate that the complement of $V_{d_G}(2)$ in ${\mathbb Z}_{\geq 0}$ is contained in 
$$L_{E_7}:= [ 0, 1, 3, 4, 7, 12, 15, 25, 28 ].
$$
\item If $G= E_8$, then $V_{d_G}(2) \subseteq {\mathbb Z}_{> 0}$.
Computations indicate that the complement of $V_{d_G}(2)$ in ${\mathbb Z}_{\geq 0}$ is contained in 
$$L_{E_8}:=
\begin{array}{c} [ 0, 2, 3, 4, 6, 8, 10, 11, 14, 16, 18, 22, 23, 24, 28, 34, 38, 40, 
\\ 46, 58, 60, 62, 88, 94, 134, 178 ].
\end{array}
$$
\end{enumerate}
\end{proposition}
\begin{proof}
Since each of the graphs $G=A_n$, $D_n$, and $E_n$, is such that
$M_G(2,\dots,2)$ is positive definite (see \ref{pro.DynkinAll} (a)), we find from \ref{emp.classical} (c) that the smallest value taken by $d_G(x_1,\dots, x_n)$  when $x_i\geq 2$ is $d_G(2,\dots, 2)$. It is classical that
 $d_{A_n}(2,\dots, 2)=n+1$, $d_{D_n}(2,\dots, 2)=4$, $d_{E_6}(2,\dots, 2)=3$, $d_{E_7}(2,\dots, 2)=2$, and $d_{E_8}(2,\dots, 2)=1$.
\end{proof}

\begin{remark} \label{ex.Dynkin4}
\begin{enumerate}[\rm (a)]
\item Let $G= D_4$.
Computations indicate that the complement of $V_{d_G}(2)$ in ${\mathbb Z}_{\geq 0}$ is contained in the following set
$$L_{D_4}:=
\begin{array}{c} [
0, 1, 2, 3, 5, 6, 7, 9, 10, 11, 13, 14, 17, 18, 19, 21, 23, 25, 26, 30, 31, 34, 35, 37, 38, 
41, \\
45, 47, 49, 53, 58, 61, 65, 66, 67, 74, 77, 79, 83, 86, 91, 93, 97, 
101, 103, 109, 
 110, \\ 114, 115, 121, 125, 
 126, 129, 130, 131, 143, 145, 153,
167, 173, 
178, 181, 187,  199, \\ 206,  210, 
223, 229, 247, 251, 258, 265, 301, 325,
343, 391, 
417, 426, 437, 451,   \\  517, 593, 595, 606, 633, 637, 
 649, 671, 763, 823,859, 871, 937, 977, \\
1027, 1087, 1330, 1517, 1661, 4477, 4585, 5273].
\end{array}
$$
\item Let $G= D_5$.
Computations indicate that the complement of $V_{d_G}(2)$ in ${\mathbb Z}_{\geq 0}$ is contained in 
$$L_{D_5}:= [ 0, 1, 2, 3, 5, 6, 7, 10, 11, 13, 15, 21, 22, 30, 31, 37, 43, 46, 55, 58, 75, 91, 102, 165, 330 ].
$$
\item Let $G= D_6$.
Computations indicate that the complement of $V_{d_G}(2)$ in ${\mathbb Z}_{\geq 0}$ is contained in 
$$L_{D_6}:=
\begin{array}{c} 
[ 0, 1, 2, 3, 5, 6, 7, 9, 11, 13, 14, 15, 17, 18, 23, 25, 27, 29, 33, 35, 38, \\
 45, 47, 49, 50, 53, 69, 71, 78, 95, 97, 105, 133, 203, 245 ].
\end{array}
$$
\item Let $G= D_7$.
Computations indicate that the complement of $V_{d_G}(2)$ in ${\mathbb Z}_{\geq 0}$ is contained in 
$$L_{D_7}:=
\begin{array}{c} 
[ 0, 1, 2, 3, 5, 6, 7, 9, 10, 13, 14, 15, 17, 19, 22, 23, 26, 27, 30, \\
33, 38, 42, 43, 49, 55, 57, 62, 78, 79, 110 ].
\end{array}
$$
\item Let $G= D_8$.
Computations indicate that the complement of $V_{d_G}(2)$ in ${\mathbb Z}_{\geq 0}$ is contained in 
$$L_{D_8}:=
\begin{array}{c} 
[ 0, 1, 2, 3, 5, 6, 7, 9, 10, 11, 13, 14, 15, 17, 18, 19, 21, 22, 25, 26, 29, 30,\\
 31, 33, 35, 37, 41, 43, 46, 49, 50,
54, 55, 58, 59, 61, 63, 65, 71, 73, 90, \\ 91, 94, 101, 105, 118, 121, 138, 169, 183, 205, 250].
\end{array}
$$
\end{enumerate}
\end{remark}

Let now $G$ be an extended Dynkin diagram. The data below for $\tilde{E_n}$, $n=6,7,8$,
suggests that the complement of the set $V_G(2)$ in ${\mathbb Z}_{\geq 0}$ might be finite.
The available data for $G= \tilde{D_n}$, $n=5,6,7$, seems to 
support the same assertion for these graphs. 
The data in the case of $G= \tilde{D_4}$, the star on $5$ vertices, is  less clear.
 
Let $G$ be the cycle ${\mathcal C}_n$, $n\geq 3$.
The available data when $n=4,5,6$ also seems to suggest that the complement of the set $V_G(2)$ is finite. The case of $G={\mathcal C}_3$ is less  clear
(see Example \ref{ex.C3}). 

\begin{proposition} \label{pro.extendedDynkin}
\begin{enumerate}[\rm (a)]
\item If $G= {\mathcal C}_n$, $n \geq 2$, then $V_{d_G}(2) \subseteq \{0\} \sqcup {\mathbb Z}_{\geq n}$,
and  $V_{d_G}(2)$ contains all positive multiples of $n$. 
\item If $G= \tilde{D_n}$, $n \geq 4$, then $V_{d_G}(2) \subseteq  \{0\} \sqcup {\mathbb Z}_{\geq 4}$.
Moreover,  $V_{d_G}(2)$ contains all positive multiples of $4$. 
\item If $G= \tilde{E_6}$, then $V_{d_G}(2) \subseteq \{0\} \sqcup {\mathbb Z}_{\geq 3}$.
Moreover,  $V_{d_G}(2)$ contains all positive multiples of $3$. 
Computations indicate that the complement of $V_{d_G}(2)$ in ${\mathbb Z}_{\geq 0}$ is contained in 
$$L_{\tilde{E_6}}:=
\begin{array}{c} 
[ 1, 2, 4, 5, 7, 8, 11, 13, 14, 16, 19, 20, 22, 23, 26, 29, 32,34, 35, 37, 41, 44, 46, 49, \\
                       53, 56, 58, 62, 71, 74, 82, 89, 95, 104, 106, 118, 128, 137, 140, 167, 172,  
                   184, 188,   \\ 212, 218, 271, 287, 302,  386 ].
\end{array}
$$
\item If $G= \tilde{E_7}$, then $V_{d_G}(2) \subseteq \{0\} \sqcup {\mathbb Z}_{\geq 2}$.
Moreover,  $V_{d_G}(2)$ contains all even positive integers. 
Computations indicate that the complement of $V_{d_G}(2)$ in ${\mathbb Z}_{\geq 0}$ is contained in 
$$L_{\tilde{E_7}}:=
\begin{array}{c} 
[ 1, 3, 5, 9, 11, 13, 15, 19, 21, 23, 25, 29, 33, 43, 45, 49, \\
51, 59, 75, 81, 115, 121, 141, 145, 159, 189 ].
\end{array}
$$
\item If $G= \tilde{E_8}$, then $V_G(2) = {\mathbb Z}_{\geq 0}$.
\end{enumerate}
\end{proposition}

\begin{proof}
Since each of the graphs $G={\mathcal C}_n$, $\tilde{D_n}$, and 
$\tilde{E_n}$, is such that
$M_G(2,\dots,2)$ is definite semi-positive,  the smallest value taken by $d_G(x_1,\dots, x_n)$  when $x_i\geq 2$ is $d_G(2,\dots, 2)=0$. 

(a) When $G={\mathcal C}_n$, the matrix $M_G(2,\dots,2)$ is the usual Laplacian. Recall that $d_G(t, 2,\dots, 2)=n(t-2)$.
We conclude from Proposition \ref{pro.arithmeticalstructure} (b) that every multiple of $n$ is in the set $V_{d_G}(2)$ and that $n$ is the smallest non-zero value in $V_{d_G}(2)$. Since the critical group of ${\mathcal C}_n$ is known to be cyclic, 
we find that $0 \in V_G(2)$.  

(b) When $G=\tilde{D_n}$, the matrix $M_G(2,\dots,2)$ is the matrix associated with an arithmetical structure with $|\Phi|=4$. The group is cyclic of order $4$ when $n$ odd, and it is $({\mathbb Z}/2{\mathbb Z})^2$ when $n$ is even. The matrix $M_G(2,\dots,2)$ is the unique arithmetical structure on $G$ of the form  $M_G(a_1,\dots,a_n)$ with $a_1,\dots, a_n \geq 2$. Thus $0 \in V_G(2)$ when $n $ is odd, and $0 \notin V_G(2)$ when $n $ is
even. Proposition \ref{pro.arithmeticalstructure} (b) shows that 
every multiple of $4$ is in  $V_{d_G}(2)$

(c) and (d) When $G=\tilde{E_n}$, $n=6$ (resp.\ $7$), the matrix $M_G(2,\dots,2)$ is the matrix associated with an arithmetical structure with $|\Phi|=3$ (resp.\ $2$). Proposition \ref{pro.arithmeticalstructure} (b) can be applied again.

(d) Since $1 \in V_{E_8}(2)$, Theorem \ref{thm.main1} (b) implies that $V_{\tilde{E_8}}(2) = {\mathbb Z}_{\geq 0}$.
\end{proof}

\begin{remark}
In view of Proposition \ref{pro.DynkinAll}, 
it is natural to wonder whether it is possible to  list all the graphs $G$
such that $M_G(3,2,\dots,2)$ is either positive definite or positive semi-definite.
Examples of such graphs which are not Dynkin diagrams and where $M_G(3,2,\dots,2)$ is positive definite can be found among the rational triple points described in \cite{Art}, page 135. We note below two constructions 
that use the classical extended Dynkin diagrams and produce infinite families of graphs $G$ where  $M_G(3,2,\dots,2)$ is positive semi-definite with determinant $0$. A more sporadic example on $5$ vertices is presented in Remark \ref{rem.0notinV}.
One may wonder whether, in the positive semi-definite case with determinant $0$, all such graphs have to appear in the list given in \cite{Ogg}.

Let $H$ be an extended Dynkin diagram on $n+1$ vertices, or a cycle. Denote by $(M_H, R_H)$ its associated arithmetical structure. Thus, $M_H=M_H(2,\dots, 2)$. Fix a vertex $v$ of $H$, and without loss of generality, assume that $v$ is the first vertex of $H$. If the coefficient of $R_H$ associated with $v$ is $1$ or $2$, we construct a graph $G$ 
on $n+3$ vertices by attaching two new vertices to $v$, each by a single edge. 
 Label the vertices of $G$ by $w,w',v_1,\dots, v_{n+1}$. We claim that $M_G(2,2,3,2,\dots, 2)$ has determinant $0$.
To show this, note the following. If $R_H=(2,\dots)$, then   
consider the vector $^t \!R_G:=(1,1, ^t \! R_H)$. It is easy to check that $M_G(2,2,3,2,\dots, 2) R_G=0$.
 If $R_H=(1,\dots)$, then   
consider the vector $^tR_G:=(1,1, 2(^t R_H))$. It is easy to check that in this case also, $M_G(2,2,3,2,\dots, 2) R_G=0$.

We illustrate this construction with two examples. 
First on the left below, we use $H=\tilde{D_6}$ to obtain a graph $G_1$ on $9$ vertices.
The old vertex of $\tilde{D_6}$ chosen for the construction is indicated in white. 
$$
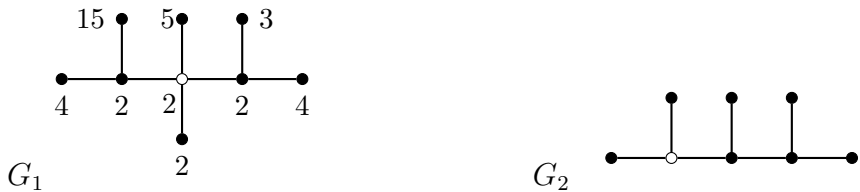
\begin{array}{cc}
G_1 \begin{tikzpicture}
[node distance=0.8cm, font=\small] 
\tikzstyle{vertex}=[circle, draw, fill, inner sep=0mm, minimum size=0.8ex]
\node[vertex]	(v1)  	at (0,1) 	[label=below:{$4$}] 		{};
 
\node[vertex]	(v33)			[right of=v1, label=below:{$2$}]	{};
 \node[vertex]	(w)			[above of=v33, label=left:{$15$}]	{};

\node[vertex, fill=none]	(v2)			[right of=v33, label=below left:{$2 \! \!$ }]	{};
\node[vertex]	(v44)			[above of=v2, label=left:{$5 \! \!$}]	{};
\node[vertex]	(v55)			[below of=v2, label=below:{$2$}]	{};
 
\node[vertex]	(v3)			[right of=v2, label=below:{$2$}]	{};
\node[vertex]	(v)			[above of=v3,  label=right:{ ${3}$}]	{};
\node[vertex]	(v4)			[right of=v3, label=below:{$4$}]	{};
 
\draw [thick] (v1)--(v33)--(v2);
\draw [thick] (v2)--(v3)--(v4);
\draw [thick] (v)--(v3);
//\draw [thick] (v44)--(v2)--(v55);
\draw[thick] (v33)--(w);
\end{tikzpicture}
\quad \quad 
\quad \quad \quad 
\quad
&
G_2 \quad \begin{tikzpicture}
[node distance=0.8cm, font=\small] 
\tikzstyle{vertex}=[circle, draw, fill, inner sep=0mm, minimum size=0.8ex]
\node[vertex]	(v1)  	at (0,1) 	[label=below:{}] 		{};
\node[vertex, fill=none]	(v33)			[right of=v1, label=below:{}]	{};
 \node[vertex]	(w)			[above of=v33, label=above:{}]	{};

\node[vertex]	(v2)			[right of=v33, label=below:{}]	{};
\node[vertex]	(v44)			[above of=v2, label=below:{}]	{};
\node[vertex]	(v3)			[right of=v2, label=below:{ }]	{};
\node[vertex]	(v)			[above of=v3,  label=right:{  }]	{};
\node[vertex]	(v4)			[right of=v3, label=below:{}]	{};
\draw [thick] (v1)--(v33)--(v2);
\draw [thick] (v2)--(v3)--(v4);
\draw [thick] (v)--(v3);
//\draw [thick] (v44)--(v2);
\draw[thick] (v33)--(w);
\end{tikzpicture}
\end{array}
$$
The graph $G_1$ is denoted (15) in Table 1, page 521 in \cite{B-D}. It is minimal in the sense that $1 \in V_{G_1}(2)$
but $1 \notin V_T(2)$ if $T$ is any subtree of $G_1$. 
Computations indicate that 
$V_{{d_{G_1}}}(2) ={\mathbb Z}_{\geq 0}$. We have labeled the graph with a set of coefficients on the diagonal that give  $\det M_{G_1}(a_1,\dots, a_9)=1$. 

The graph $G_2$ above on the right has $8$ vertices and is obtained from  the construction with   $H=\tilde{D_5}$.
The old vertex of $\tilde{D_5}$ chosen for the construction is indicated in white. 
The graph $G_2$ is a subgraph of $G_1$. 
Computations indicate that the complement of $V_{d_{G_2}}(2)$ in 
${\mathbb Z}_{\geq 0}$ is contained in 
$$L_{G_2}= \begin{array}{c}
[ 1, 5, 23, 25, 31, 53, 61, 71, 73, 145, 163, 199, \\ 211, 
229, 275, 289, 365, 379,383, 421, 451, 493, 799, 1153 ].
\end{array}.$$

A different construction is as follows. Start with an extended Dynkin diagram $H=\tilde{D_n}$ with $n\geq 4$.
Fix a pair of vertices of degree $1$ attached to the same vertex  in $H$.
Without loss of generality, we can assume that these vertices of degree $1$ are $v_1$ and $v_2$, attached to a vertex $v_3$. Consider the graph $G$ obtained by linking a new vertex $w$ to $v_1$ by a single edge. Order the vertices of $G$ as $w,v_1, v_2,v_3,\dots $. Then 
we claim that $M_G(2,2,3,2,\dots, 2)$ has determinant $0$. 
The vector $^tR_G:= (2,4,2,6,\dots,6,3,3)$ is such that $M_G R_G=0$. Here we have ordered the vertices of $H$ such that the last two vertices again have degree $1$, and the antepenultimate vertex has degree $3$ in $H$.
When $n=4$, note that all vertices of degree $1$ are linked to the same vertex of degree $4$. The construction produces the graph $
S_5^+$ in this case

We illustrate this second construction with the example of $H=\tilde{D_5}$, obtaining a graph on $7$ vertices.
The old vertex of $\tilde{D_5}$ chosen for the construction is indicated in white: 
$$
G_3 \quad 
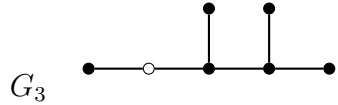
\begin{tikzpicture}
[node distance=0.8cm, font=\small] 
\tikzstyle{vertex}=[circle, draw, fill, inner sep=0mm, minimum size=0.8ex]
\node[vertex]	(v1)  	at (0,1) 	[label=below:{}] 		{};
\node[vertex, fill=none]	(v33)			[right of=v1, label=below:{}]	{};

\node[vertex]	(v2)			[right of=v33, label=below:{}]	{};
\node[vertex]	(v44)			[above of=v2, label=below:{}]	{};
\node[vertex]	(v3)			[right of=v2, label=below:{}]	{};
\node[vertex]	(v)			[above of=v3,  label=right:{ ${}$}]	{};
\node[vertex]	(v4)			[right of=v3, label=below:{}]	{};
\draw [thick] (v1)--(v33)--(v2);
\draw [thick] (v2)--(v3)--(v4);
\draw [thick] (v)--(v3);
//\draw [thick] (v44)--(v2);
\end{tikzpicture}
$$
The graph $G_3$  is a subgraph of the previous examples. 
Computations indicate that the complement of $V_{d_{G_3}}(2)$ in 
${\mathbb Z}_{\geq 0}$ is contained in $L_{G_3}=[ 1, 21, 25, 37, 75 ]$.
\end{remark}

\if false

a similar proposition can be proved for matrices of the form $M_G(3,2,\dots,2)$.  As 
the set of matrices of the form $M_G(3,2,\dots,2)$ which are positive semi-definite would need to be listable.

Examples of arithmetical structures 

Cycle with a -3 curve, and two -2 curves attached to this vertex

Extended Dynkin D_n  2-4-6-...-6-3 (where usual extended Dynkin is 1-2---2-1
The only -3 curve meets the 6 in 2-4-6, giving it multiplicity 2.

On the cycle on 4 vertices: add a -3 anywhere and get a structure with positive definite. Can we extend this to a positive semi-definite? 
Yes if we add a -2 curve attached to the curve opposite to the -3. 
Otherwise, attaching a -2 somewhere else still makes it positive definite.   

\end{remark}

\fi

\end{section}

\begin{section}{The graphs $A_n$ and small variations} \label{sec.paths}

We discuss in this section the paths $A_n$, $n\geq 3$, and some graphs $A_3(e,f)$ on three vertices generalizing $A_3$. We also consider the cycle ${\mathcal C}_3$ and generalizations at the end of the section. Let us note here again that when $G=A_n$, then $0 \notin V_{d_G}(2)$ since the matrix $M_G(2,\dots,2)$ is positive definite. We also note that the group $\Phi$ associated to any matrix of the form $M_G(a_1,\dots, a_n)$ with $a_1,\dots, a_n \geq 2$ is always cyclic (see \cite{Lor13}, Lemma 3.13).
Thus we have $V_G(2)=V_{d_G}(2)$ when $G=A_n$. 

Theorem \ref{thm.density} shows that the complement of $V_{A_n}(2)$ in ${\mathbb Z}_{\geq 0}$ is dense.
The data below in Example \ref{ex.paths} suggests that this complement might always be finite. The smallest values in the complement can be explicitly described using the next lemma, and there are at least $4n-6$ such values.

\begin{lemma} \label{lem.Anmissingvalues}
 Let $G=A_n$, $n \geq 3$.  The set
$V_G(2)$ starts with the following values:
$$[ n+1, 2n+1, 3n-1, 3n+1, 4n-5, 4n, 4n+1, 5n-11, 5n+1, \dots
].$$
\end{lemma}

\begin{proof}
We have  
$d_{A_n}(t,2,\dots, 2)= n(t-1) +1$. This shows that $V_G(2)$ contains $n+1,2n+1, 3n+1,\dots$. 
We have $d_{A_n}(2,\dots,2, 3,2)=3n-1$ and  $d_{A_n}(3,2,\dots, 2,3)=4n$.
Moreover, $d_{A_n}=(2,2,\dots, 3,2,2)=4n-5$. More generally, placing $3$ on the $i$-th column from the end (with $2i\leq n$):
the determinant is $n+1$ + $i(n-i+1)= (i+1)n-(i^2-i-1)$.
With $i=4$, we obtain $5n-11$.
\end{proof}

\begin{example} \label{ex.paths}
\begin{enumerate}[\rm (a)]
\item 
Computations indicate that the complement of $V_{A_3}(2)$ in ${\mathbb Z}_{\geq 0}$ is contained in the following set
$$L_{A_3}:= [ 0,1, 2, 3, 5, 6, 9, 11, 14, 15, 35, 105, 510 ].
$$
\item 
Computations indicate that the complement of $V_{A_4}(2)$ in ${\mathbb Z}_{\geq 0}$ is contained in 
$$L_{A_4}:=
\begin{array}{c} 
[ 0, 1, 2, 3, 4, 6, 7, 8, 10, 12, 14, 15, 20,  22, 24, 26, 28, 38, 40, 42, 48, 52,
68,\\ 104, 132, 150, 188, 314 ].
\end{array}
$$
\item 
Computations indicate that the complement of $V_{A_5}(2)$ in ${\mathbb Z}_{\geq 0}$ is contained in 
$$L_{A_5}:=
\begin{array}{c} 
[0, 1, 2, 3, 4, 5, 7, 8, 9, 10, 12, 13, 17, 18, 19, 27, 28, 34, 40, 52, 63, 88 ].
\end{array}
$$
\item 
Computations indicate that the complement of $V_{A_6}(2)$ in ${\mathbb Z}_{\geq 0}$ is contained in 
$$L_{A_6}:=
\begin{array}{c} 
[ 0, 1, 2, 3, 4, 5, 6, 8, 9, 10, 11, 12, 14, 15, 16, 18, 20, 21, 22, 23, 26, 29,
30, 32, 36, 38, \\ 42,
 44, 48, 52, 54, 56, 62, 70, 80, 81, 86, 96, 102, 108, 110,
122, 126, 140, 180, 236 ].
\end{array}
$$
\end{enumerate}
\end{example}

It is natural to wonder whether $d_{A_3}(x,y,z)=xyz-x-z$ takes every  non-negative value when $x,y,z \geq 2$, except for the values listed in $L_{A_3}$. We have not been able to answer this question beyond the following remarks.

\begin{proposition} \label{pro.density4}
The set $V_{A_3}(2)$ contains: 
\begin{enumerate}[\rm (a)]
\item All even positive integers $n$, except possibly  those of the form $n=2^m-2$ with $m$ odd or $m=2,4$.
\item All odd positive integers $n$ such that $n+2$ is not prime.
\item All odd positive integers
congruent to $1$ modulo $4$ and  congruent to $2$ or $8$ modulo $9$. 
\end{enumerate}
\end{proposition}

\begin{proof} 
We claim that  the complement of $V_{A_3}(2)$ in ${\mathbb Z}_{\geq 0}$ consists only of integers $n$ such that {\rm (i)} 
either $n+2$ is prime or $n+2$ is a power of $2$,
and {\rm (ii)} such that $n+4$ is prime, or $n+4$ is not divisible by a prime $p \geq 7 $ congruent to $3$ modulo $4$.

Indeed, writing $xyz-x-z=n$, we find that
when $z=2$, we have $x(2y-1)=n+2$. This can be solved with $x,y \geq 2$ when $n+2$ is not prime, and not a power of $2$. When $z=4$, we have $x(4y-1)=n+4$. This can be solved with $x,y \geq 2$ when $n+4$ is not prime and when at least one divisor of $n+4$ is congruent to $3$ modulo $4$ and greater than $3$. 

Suppose now that $m > 4$ is even. Set $x=4$, $z=6$, and $y=(2^{m-3}+1)/3 \geq 2$.
It is clear that $xyz-x-z = 2^m-2$. When $m \geq 4$ is even, $2^{m-3}+1$ is always divisible by $3$.

Suppose now that $n$ is congruent to $1$ modulo $4$ and 
 congruent to $2$ or $8$ modulo $9$. Then in particular $n$ is not divisible by $3$.
Hence,   one of $n+2$ or $n+4$ has to be divisible by $3$. Suppose that $n+2$ is divisible by $3$. Then $n$ is in $V_{A_3}(2)$ since $n+2$ is neither prime nor  a power of $2$.  Suppose now that $n+4$ is divisible by $3$. Then $n$ is in $V_{A_3}(2)$ unless $(n+4)/3$ is only divisible by primes congruent to $1$ modulo $4$ and by a power of $3$.  Our hypothesis implies that $n+4 $ is congruent to $3$ or $6$ modulo $9$
and so is exactly divisible by $3$.
Since the power of $3$ is odd in $n+4$, then  $n$ has to be congruent to $3$ modulo $4$ when 
$(n+4)/3$ is only divisible by primes congruent to $1$ modulo $4$. 
Since we assume that $n$ is congruent to $1$ modulo $4$, we find that $n\in V_{A_3}(2)$.
\end{proof}

\begin{remark} The first values of $2^{2m+1}-2$ are $0, 6, 30, 126, 510, 2046, 8190,\dots$. 
The values $0,6$, and $510$ are not achieved
by $xyz-x-z$ with $x,y,z\geq 2$. The values $30$ and $126$ are achieved exactly once, with $(x,y,z)= ( 3, 4, 3)$ and $ (12, 2, 6)$, respectively.
\end{remark} 

Consider more generally the polynomial $f(x,y,z)=xyz-ax-bz$ with $a,b \in {\mathbb Z}_{\geq 1}$. It is clear that $V_f(1)={\mathbb Z}_{\geq 0}$ since $f(1,N+a+b,1)=N$ for any integer $N$.

{\it If an integer $p$ divides $a$, then $V_{f}(2)$
contains all non-negative multiples of $p$.} Indeed,
assume that $p$ divides $N$.
 Set $z=p$,  $y=a/p+1$, and $x=N/p+b\geq 2$. Then $xyz-ax-bz=N$. 

\begin{proposition} \label{pro.G(a,b)}
Let $f(x,y,z)=xyz - ax   - bz$ with $a \geq b\geq 1$. 
We have $V_{f}(2)\supseteq  {\mathbb Z}_{>0}$ in the following cases: 
\begin{enumerate}[\rm (a)]
\item $a+1$ or  $b+1 $ is not prime.
 \item $a$ is divisible by $4$ and $b=1$.
\end{enumerate}
Moreover, $0 \in V_{f}(2)$ when $(a,b) \neq (1,1), 
(2,1)$, and $(4,1)$.
\end{proposition}

\begin{proof}
(a)
Without loss of generality, we can assume that $b+1$ is not prime.
 Consider the equation $xyz-ax-bz=N$, which we rewrite as $(xy-b)z=N+ax$. Under our hypothesis, the equation $xy-b=1$ 
can be solved with $x,y \geq 2$, and so setting $z=N+ax \geq 2$ shows that $N \in V_{f}(2)$.

(b) 
Assume now that $b=1$.
If $N$ is even, we use the case $p=2$ just above to conclude that since $2$ divides $a$, $xyz-ax-z$ takes all possible even values when $x,y,z \geq 2$.
If $N$ is odd and  $a=4c$,  set $y=2$ and 
$z=2c+1 \geq 2$. Since $z$ is then odd and $N$ is odd, we can set $x=(N+z)/2 \geq 2$, to get $xyz-4cx-z=N$.

The equation $xyz-ax-bz=0$ always has the solution $(x,y,z)=(b,2,a)$, so that $ 0 \in V_{f}(1)$, and when $a,b>1$, 
$ 0 \in V_{f}(2)$. 
Suppose now that $b=1$, and let us show that we can solve
$xyz-ax-z=0$ with $x,y,z \geq 2$ if and only if $a \neq 1,2,4$. If we can solve this equation, then $(xy-1)z=ax$ and $x$ must divide $z$. Write $z=cx$. We need to solve  $(xy-1)c=a$, so $c$ divide $a$.
This equation can be solved if we can find integers $x,y\geq 2$ such that
$xy=1+a/c$.
We claim that given any integer $m>1$, $m\neq 2,4$, we can find a divisor $d$ of $m$ such that $d+1$ is not prime. This is clear if $m$ is divisible by an odd prime. 
When $m$ is a power of $2$, this is true as soon as $8$ divides $m$. We apply this claim to $a$, and we find $c$ which divides $a$ such that $1+a/c$ is not prime, unless $a=1,2,4$, as desired.
\end{proof}
\begin{example}
Computations indicate that in the case of $f=xyz-2x-z$, the complement in 
${\mathbb Z}_{\geq 0}$ of the set $V_f(2)$ might be reduced to $\{0,1,3,7,15\}$.
\end{example}

\begin{remark} \label{rempro.G(a,b)}
Fix $e_1,\dots, e_{n-1} \geq 1$. 
Consider  the graph $A_n(e_1,\dots, e_{n-1})$ with multiple edges  obtained as follows. 
Given $n$ vertices $v_1,\dots, v_n$, link $v_i$ to $v_{i+1}$ by $e_i$ edges, for each $i=1,\dots, n-1$.

The matrix $M_{A_n(e_1,\dots,e_{n-1})}(x_1,\dots,x_n)$ is a tridiagonal matrix, and such matrices occur in the following context.
Recall that two symmetric matrices $M$ and $N$ in $M_n({\mathbb Z})$ are {\it congruent} if there exists $T \in {\rm GL}_n({\mathbb Z})$ such that $N= (^tT)MT$. Newman (\cite{New}, Theorem 1) showed that any positive definite matrix  $M \in M_n({\mathbb Z})$ is congruent to a tridiagonal matrix
with certain specified properties. When $\det(M)=1$, the tridiagonal matrix  is of the form 
 $M_{A_n(1,\dots,1,e_{n-1})}(a_1,\dots, a_n)$. Such matrices are further studied in \cite{DG} and \cite{Ger}. 
 
 In the case $n=3$, the graph $G=A_3(e,f)$ has matrix  $M_{A_3(e,f)}(x,y,z)$ with determinant 
 $d_G(x,y,z)=xyz-f^2x-e^2z$. As we saw in Proposition \ref{pro.G(a,b)},
 the determination of $V_{d_G}(2)$ is not immediate when both $e^2+1 $ and $f^2+1$ are prime.
  
  Let $G=A_3(2,1)$, with
$d_G(x,y,z)=xyz-x-4z$, and  $d_G(3,2,2)=1$. 
  The graph $G$ is the graph with the fewest vertices and edges such that  $1 \in V_G(2)$. See Examples 
 \ref{ex.lollipop} and \ref{ex.cone} for examples on $4$ vertices.
 Proposition \ref{pro.G(a,b)} shows that $V_{d_G}(2)={\mathbb Z}_{>0}$.
 
 Let $G=A_3(2,2)$, with
$d_G(x,y,z)=xyz-4x-4z$. The set $V_{d_G}(2)$ contains all even positive integers, but it is not immediately clear from computations that the complement of $V_{d_G}(2)$
 in ${\mathbb Z}_{\geq 0}$ is finite.
The situation is similar for $V_h(2)$ when $h:=xyz-2x-2z$. The integer $N=538641$ belongs to the complement of both  $V_{d_G}(2)$ and $V_h(2)$.
 
\end{remark}  

\if false

\begin{example} \label{ex.P3}
Given $a,b \in {\mathbb Z}_{>0}$, consider the graph $G=G(a,b)$ on $n=3$ vertices with matrix
$$M_G=\left( \begin{array}{ccc}
z&-a&0  \\
-a&y&-b\\
0&-b&x
\end{array}
\right) \mbox{ and } d_G(x,y,z)=xyz - a^2x   - b^2z.$$
The matrix $M_G(a_1,a_2,a_3)$ with $a_1,a_2,a_3>0$ is  positive definite (resp.\ positive semi-definite) if and only if  $d_G(a_1,a_2,a_3)>0$ (resp.\ $d_G(a_1,a_2,a_3)\geq 0$).
The graph $G(1,1)$ is nothing by the path $A_3$ on $3$ vertices.

The banana graph on $2$ vertices and $a$ edges is an induced subgraph $H$ with $d_H(t,c)=ct-a^2$. Taking $c \geq 2$ and coprime to $a$, we find from  Theorem  \ref{thm.density} (b) that 
$V_{d_{G(a,b)}}(2)$ is dense in ${\mathbb Z}_{\geq 0}$. Our next proposition shows that $V_{d_G}(2)={\mathbb Z}_{\geq 0}$
for most   pairs $(a,b)$.

\if false
Note that if a prime $p$ divides both $a$ and $b$, then $V_{d_{G}}(2)$
contains all non-negative multiples of $p$. Indeed,
assume that $p$ divides $N$ and $a=pc$, $b=pd$.
 Set $y=p\geq 2$, $x=pd^2+1\geq 2$ and $z=N/p+pc^2x \geq 2$. Then $xyz-p^2c^2x-p^2d^2z=N$. 
\fi
 
 Note that if an integer $p$ divides $a$, then $V_{d_{G}}(2)$
contains all non-negative multiples of $p$. Indeed,
assume that $p$ divides $N$ and $a=pc$.
 Set $z=p$,  $y=pc^2+1$, and $x=N/p+b^2\geq 2$. Then $xyz-p^2c^2x-b^2z=N$. 
\end{example}

\begin{remark} Let $G=G(1,2)$.
Note that $d_G(x,y,z)=xyz-x-4z$, and that $d_G(3,2,2)=1$. 
  The graph $G$ is the graph with the fewest vertices and edges such that  $1 \in V_G(2)$. See Examples 
 \ref{ex.lollipop} and \ref{ex.cone} for examples on $4$ vertices.
\end{remark}  

\if false 
Assume from now on that $N$ is even. If $N$ is of the form $N=2(2M+1)$, 
set  $x=z=2$, and consider the resulting equation $4y-8c^2-2=N$.
This equation is solvable if $N+8c^2+2$ is divisible by $4$.
This is the case under our hypothesis since it implies that  $4$ divides $N+2$. Set $y= (N+8c^2+2)/4 \geq 2$. 

If $N$ is of the form $N=4(3Q+1)$, set  $x=3$ and $z=2$, and consider the resulting equation $6y-12c^2-2=N$.
This equation is solvable if $N+12c^2+2$ is divisible by $6$.
This is the case under our hypothesis since it implies that $6$ divides $N+2$.

If $N$ is of the form $N=4(3Q+2)$, set  $x=3$ and $z=4$, and consider the equation $12y-12c^2-4=N$.
This equation is solvable if $N+12c^2+4$ is divisible by $12$.
This is the case under our hypothesis since it implies that   $N+4$ is divisible by $12$.

If $N$ is of the form $N=12Q$, set  $x=3$  and consider the equation $3yz-12c^2-z=N$, which we rewrite
as $(3y-1)z =12(Q+ c^2)$.  If $Q+ c^2+1 $ is divisible by $3$, set $z=12$ and $y=(Q+ c^2+1)/3$.
If $2(Q+ c^2)+1 $ is divisible by $3$, set  $z=6$ and $y=(2(Q+ c^2)+1)/3$.
If $(Q+ c^2)$ is divisible by $3$, factor $ Q+ c^2 = t Q'$, where $t$ is a power of $3$ and $Q'$ is coprime to $3$.
If $Q'+1 $ is divisible by $3$, set $z=12t$ and $y=(Q'+1)/3$.
If $2Q'+1 $ is divisible by $3$, set  $z=6t$ and $y=(2Q'+1)/3$. We need $y=(Q'+1)/3 \geq 2$ in the first case, 
and $y=(2Q'+1)/3 \geq 2$ in the second case. These inequalities occur
unless $Q'=2$ in the first case, and $Q'=1$ in the second case.

In the special case where $Q'=2$ and $N=12(2 \cdot 3^m-c^2)$, set $x=3$, $y=3$, and $z=3^{m+1}$.

If $N$ is of the form $N=12Q$, set  $z=12$  and consider the equation $12yx-4c^2x-12=N$, which we rewrite as $x(3y-c^2) =3(Q+1)$. This equation is solvable if $3$ divides $c$, by setting $x=Q+1$ and $y=c^2/3$.

If $N$ is of the form $N=12Q$, set  $z=4$  and consider the equation $4yx-4c^2x-4=N$, which we rewrite as $x(y-c^2) =3Q+1$. This equation is solvable if $3Q+1$ is not prime. The integer $3Q+1$ is even when $Q$ is odd, and when $Q$ is of the form $3^m-c^2$, we find that $Q$ is odd when $c$ is even.
\fi

\begin{lemma}  Let $G=G(a,b)$ be the graph introduced in \ref{ex.P3}.
In the cases below where $a=b$ and $a^2+1$ is prime\footnote{It is conjectured, but not known, that there exist infinitely many integers $a$ such that $a^2+1$ is prime (see \cite{AFG}, 3.3).}, we have
\begin{enumerate}[\rm (i)] 
\item $1 \notin V_{d_G}(2)$ when $(a,b)=(1,1)$, $(2,2)$ or $(6,6)$.
\item $1 \in V_{d_G}(2)$ when $(a,b)=(4,4)$, $(10,10)$,  $(14,14)$ or $(26,26)$. 
\end{enumerate}
\end{lemma}
\begin{proof}
(i) When $(a,b)=(1,1)$, use Lemma \ref{lem.Anmissingvalues}, or argue as follows. Consider the equation $xyz-x-z=N$, which we rewrite as $xz(y-1) +(x-1)(z-1)= N+1$. When $x,y,z>1$, we find that 
$ xz(y-1) +(x-1)(z-1)\geq 5$. Thus the equation cannot be solved when $N \leq 3$.

 Consider the equation $xyz-4x-4z=1$, which we rewrite as $xz(y-4) +4(x-1)(z-1)= 5$. When $x,y,z>1$, we find that 
$ xz(y-4) +4(x-1)(z-1)> 5$ unless $y=2,3,4$.   The equation  $xyz-4x-4z=1$ is impossible when $y$ is even.
When $y=3$, we have $3xz-4x-4z=1$. It follows that $x=4X-3$ and $z=4Z-1$ with $X,Z \geq 1$, leading to the equation
$12XZ - X + 5Z =2$. It is clear that $(12X+5)Z = X+2$ is impossible.

 Consider the equation $xyz-36x-36z=1$, which we rewrite as $xz(y-36) +36(x-1)(z-1)= 37$. When $x,y,z>1$, we find that 
$ xz(y-36) +36(x-1)(z-1)\geq 40$ unless $y=2,\dots,36$.   
The equation  $xyz-36x-36z=1$ is impossible when $y$ is a multiple of $2$ or $3$.  This leaves $y=5, 7, 11, 13, 17, 19, 23, 25, 29$.
Check that for each such $y$ and for each $x \in [2,\dots,36]$, the 
equation $xyz-36x-36z=1$ cannot be solved. Then 
the equation has no solutions, because if both $x,z >36$, 
the equation $y -36/z-36/x=1/xz$ is not solvable.

(ii) When $a=4$, take $(x,y,z)= (5,5,9)$. 
When $a=10$, take $(x,y,z)= (17,17,9)$.   
When $a=14$, take $(x,y,z)= (169,9,25)$.  
When $a=26$, take $(x,y,z)= (5,149,49)$. 
\end{proof}

\begin{lemma} \label{lem.G(a,b))}
Let $G=G(a,b)$ be the graph introduced in \ref{ex.P3}. 
\begin{enumerate}[\rm (a)]
\item Assume that $(x_0,y_0,z_0) \in {\mathbb Z}_{>0}^3$ is such that $d_G(x_0,y_0,z_0)=0$.
Let $\delta:=\gcd(x_0,z_0)$. Then $(x_0/\delta,\delta y_0,z_0/\delta)$ is also a solution of the equation $d_G(x,y,z)=0$.
When $\delta=1$, then $x_0$ divides $b^2$ and $z_0$ divides $a^2$.

\item Let $x_0$ be any positive divisor of $b^2$. Let $z_0$ be any positive divisor of $a^2$. Let $y_0=a^2/z_0+b^2/x_0$. Then 
$(x_0,y_0,z_0)$ is such that $d_G(x_0,y_0,z_0)=0$.

\item 
We have $0 \in  V_{d_{G(a,b)}}(1)$, and when $a,b>1$,  $0 \in  V_{d_{G(a,b)}}(2)$. 
If either $a$ or $b$ equals $1$, 
then $0 \in V_{d_{G(a,b)}}(2)$ except when $(a,b)=(1,1)$ or $(1,2)$.

\item We have $0 \in  V_{{G(a,b)}}(1)$, and when $a,b>1$,  $0 \in  V_{{G(a,b)}}(2)$.
\if false
Let $(x_0,y_0,z_0) \in {\mathbb Z}_{>0}^3$ be such that $d_G(x_0,y_0,z_0)=0$. If $\gcd(x_0,z_0)=1$ or $a$ and $b$ are not both even,
then the group $\Phi$ associated with the matrix $M_G(x_0,y_0,z_0)$ is cyclic. In particular, 
\fi
\end{enumerate}

\end{lemma}
\begin{proof} The proofs of (a) and (b) are left to the reader. 

(c) We always have the solution $x=b^2$, $y=2$, and $z=a^2$. So $ 0 \in V_{d_{G(a,b)}}(1)$, and when $a,b>1$, 
$ 0 \in V_{d_{G(a,b)}}(2)$. 
Suppose now that $a=1$, and let us show that we can solve
$xyz-x-b^2z=0$ with $x,y,z \geq 2$ if and only if $b>2$. If we can solve this equation, then $z$ must divide $x$. Write $x=cz$. We need to solve  $cyz-c-b^2=0$, so $c$ divide $b^2$.
This equation can be solved if we can find integers $y,z\geq 2$ such that
$yz=1+b^2/c$.
We claim that given any integer $m>1$, $m\neq 2,4$, we can find a divisor $d$ of $m$ such that $d+1$ is not prime. This is clear if $m$ is divisible by an odd prime. 
When $m$ is a power of $2$, this is true as soon as $8$ divides $m$. We apply this claim to $b^2$, and we find $c$ which divides $b^2$ such that $1+b^2/c$ is not prime, unless $b^2=1,4$, as desired.

(d) Consider again  the solution $x=b^2$, $y=2$, and $z=a^2$.
The associated matrix 
is  positive semi-definite. The associated group $\Phi$ is cyclic 
when $a$ and $b$ are not both even. To see this, note that the group $\Phi$ in this case needs at most two generators, and is cyclic if and only of the greatest common divisor of the coefficients of the matrix is equal to $1$. Since $y=2$ is a coefficient of our matrix, it is coprime to $a$ or $b$. Thus, $0 \in V_{G(a,b)}(1) $, and  $0 \in V_{G(a,b)}(2) $ when $a,b>1$ are not both even.

 Suppose that $a$ and $b$ are both even. 
When both are powers of $2$, choose $x_0=b^2$ and $z_0=a^2/2>2$. Then $y_0= a^2/z_0+b^2/x_0=3$ is coprime to $2$.
It follows that the $\gcd$ of the coefficients of $M_G(x_0,y_0,z_0)$ is $1$, and the associated group $\Phi$ is cyclic.
When there exists an odd prime $p$ which divides $a$, choose $x_0=2$ and $z_0=p>2$. Then $y_0= a^2/z_0+b^2/x_0>2$.
\end{proof}
\fi



\begin{example} \label{ex.K3} \label{ex.C3}
Consider the cycle $G={\mathcal C}_3$ on $n=3$ vertices. Then 
$$
\if false
M_G=\left( \begin{array}{ccc}
x&-1&-1  \\
-1&y&-1\\
-1&-1&z
\end{array}
\right)\quad \text{and } \quad 
\fi 
d_G(x,y,z)=xyz - x - y - z - 2.$$
We noted in Proposition \ref{pro.extendedDynkin} (a) that $V_{d_G}(2) $ contains all multiples of $3$.
\if false
Since $d_G(2,2,z)=3(z -2)$, the set $V_{d_G}(2) \cap {\mathbb N}$ contains all positive multiples of $3$.
Since  $d_G(2,3,z+2)=5z+3$, the set $V_{d_G}(2) \cap {\mathbb N}$ contains all positive integers congruent to $3$ or $8$ modulo $10$.
Since  $d_G(3,3,z+2)=8z+8$, the set $V_{d_G}(2) \cap {\mathbb N}$ contains all positive multiples  of $8$.

We have   $d_G(4,4,z)=5(3z -2)$, so that $V_{d_G}(2) \cap {\mathbb N}$ contains all positive multiples $k$ of $5$ such that $k+1$ is divisible by $3$.
We have $d_G(4,9,z)=5(7z -3)$, so that 
when $z=3u+2$, $V_{d_G}(2) \cap {\mathbb N}$ contains the multiple $k=5(7z -3)$ of $5$  such that $k+2$ is divisible by $3$.
\fi
It is not computationally clear in this example that the complement 
of $V_{d_G}(2)$ in ${\mathbb Z}_{\geq 0}$
is finite. This complement is likely to contain $2201$ values in the interval  $[1, 2 \cdot 10^6]$.
\if false
 ($MV$ stands for {\it missing values}). We computed the set $MV_{d_G}(2) \cap \{1,\dots,N\}$ explicitly up to $N=2000000$. In the table below, we include in Column $2 $
the size of the set $MV_{d_G}(2) \cap \{1,\dots,N\}$ up to the $N$ given in Column $1$. The column $\#MV2$ gives the number of even numbers in that set.
Similarly the column $\#MV5$ (resp.\ $\#MVP$) gives the number of multiples of $5$ in that set (resp.\ the number of prime numbers in the set). As we noted in Proposition \ref{pro.complexity} (a),
the set $MV_{d_G}(2)$ does not contain any multiples of $3$ since $3$ is the number of spanning trees of ${\mathcal C}_3$.

 \renewcommand{\arraystretch}{1.3}
$$
\begin{tabular}{|c|c|c|c|c|}
\hline
$N$ & $\#MV$ & $\#MV2$   & $\#MV5$  &  $\#MVP$ \\
\hline
1000000 & 1791 & 261  (14.57\%) & 280 (15.63 \%) & 378 (21.11\%)\\
\hline
1500000 & 2036 & 283 (13.90\%) & 319 (15.67\%) & 427 (20.97\%)\\
\hline
2000000 & 2201 & 293 (13.31\%) & 344 (15.63\%) & 473 (21.49\%) \\
\hline
\end{tabular}
$$
\smallskip
\fi

\if false
Indeed, up to $N=700000$, there are $1610$ values in the set  $MV_{d_G}(2)$, with $348$ primes among them.

Indeed, up to $N=800000$, there are $1678$ values in the set  $MV_{d_G}(2)$, with $360$ primes among them.

Indeed, up to $N=900000$, there are $1732$ values in the set  $MV_{d_G}(2)$, with $370$ primes among them.

Indeed, up to $N=1000000$, there are $1791$ values in the set  $MV_{d_G}(2)$, with $378$ primes among them.

Indeed, up to $N=1100000$, there are $1847$ values in the set  $MV_{d_G}(2)$, with $391$ primes among them.

Indeed, up to $N=1200000$, there are $1897$ values in the set  $MV_{d_G}(2)$, with $400$ primes among them.

Indeed, up to $N=1300000$, there are $1943$ values in the set  $MV_{d_G}(2)$, with $406$ primes among them.

Indeed, up to $N=1500000$, there are $2036$ values in the set  $MV_{d_G}(2)$, with $427$ primes among them.

Indeed, up to $N=1800000$, there are $2128$ values in the set  $MV_{d_G}(2)$, with $440$ primes among them.

Indeed, up to $N=2000000$, there are $2201$ values in the set  $MV_{d_G}(2)$, with $473$ primes among them.

More precisely, $MV_{d_G}(2)$ up to $N$ contains $268$ values $k$ congruent to $5$ modulo $10$ and such that $k+2$ is divisible by $3$. It also contains $54$ values $k$ divisible by $10$ and such that $k+2$ is divisible by $3$.
\fi

Theorem \ref{thm.main1}  shows that $V_{d_G}(1)= {\mathbb Z}_{\geq 0}$. 
The recent preprint \cite{CS} studies the set $V_f(1)$ of values taken by the related polynomial 
$f(x,y,z)=xyz+x+y+z$. The authors suggest on page $3$ that the complement of $V_f(1)$ in ${\mathbb Z}_{\geq 0}$ is infinite.
\if false
We expect that this estimate is far from optimal. We did calculations for N up to the 250
millionth prime (5336500537) and found 2014 prime numbers which are not the sum plus
the product of three positive integers.
\fi

 Consider more generally the polynomial $f(x,y,z)=xyz-ax-by-cz$, where $a \geq b \geq c >0$ are fixed integers. {\it Then $V_f(2) ={\mathbb Z}_{\geq 0}$ if one of the integers $a+1$, $b+1$, or $c+1$, is not prime.} Indeed, assume without loss of generality 
 that $c+1$ is not prime. Then the equation $xy-c=1$ can be solved with integers $x_0,y_0 \geq 2$. For any integer $N\geq 0$, set $z_0:=N+ax_0+by_0$. Then $f(x_0,y_0,z_0)=N$. 
 
When $(a,b,c)=(2,1,1)$, the set of values in the complement of $V_f(2)$ in ${\mathbb Z}_{\geq 0}$  might consist only of $ \{ 1, 2, 5, 61 \}$.
 When $(a,b,c)=(2,2,1)$, the set of values in the complement of $V_f(2)$ in ${\mathbb Z}_{\geq 0}$  might consist only of $\{3, 11, 27, 251\}$.
 When $(a,b,c)=(10,1,1), (6,1,1),$ or $(4,2,1)$, it is likely that $V_f(2)={\mathbb Z}_{\geq 0}$.

Add  edges to ${\mathcal C}_3$ to obtain the graph $G={\mathcal C}_3(e_1,e_2,e_3)$   with matrix
$$M_G=\left( \begin{array}{ccc}
x&-e_3&-e_2  \\
-e_3&y&-e_1\\
-e_2&-e_1&z
\end{array}
\right) \mbox{ and } d_G(x,y,z)=xyz - e_1^2x  -e_2^2y-e_3^2z-2e_1e_2e_3.
$$
When $(e_1,e_2,e_3)=(2,1,1)$, computations indicate that the set of values in the complement of $V_{d_G}
 (2)$ in ${\mathbb Z}_{\geq 0}$ is very small
 and might consist only of $\{8, 56, 248\}$.
 
 When $(e_1,e_2,e_3)=(2,2,2)$, the set $V_{d_G}
 (2)$ contains all even values. Indeed, $(x,y,z)=(2,4,z)$ produces all values divisible by $4$, and $(x,y,z)=(2,3,z)$ produces all values of the form $4m+2$. 
It is not computationally clear in this case that the set of values in the complement of $V_{d_G}
 (2)$ in ${\mathbb Z}_{\geq 0}$ is finite. Theorem \ref{thm.density} (b) shows that
  $V_{d_G}
 (2)$ is dense in ${\mathbb Z}_{\geq 0}$. 
 \end{example}
 
\if false
\begin{example} \label{ex.C4}
Consider the cycle $G$ on $n=4$ vertices. 
\if false
Then
$$M_G=\left( \begin{array}{cccc}
w&-1&0&-1 \\
-1&x&-1&0\\
0&-1&y&-1\\
-1&0&-1&z
\end{array}
\right)$$ 
and $$d_G(w,x,y,z)=
w x y z - w x - w z - x y - y z.$$
Since $d_G(w,2,2,2)=4w-8$, we find that 
$V_{d_G}(2)$ contains all positive multiples of $4$.
\fi
Computations produced a set of $325$ positive integers up to $10^6$ which 
contains the complement of $V_{d_G}(2) \cap [1,10^6]$. The four largest values in that set are 
$86899,  184549, 204997$,  and $858811$. The paucity of large values in this set suggests that the set of missing values might be finite.

Consider now the cycle $G$ on $n=5$ vertices. Computations produced a set of $123$ values up to $10^5$
which contains the complement of $V_{d_G}(2)$ in $[1,10^5]$. The largest two values in this set are $5422$ and $ 12489$.
This data suggests that the set of missing values might be finite in this case also.
\end{example}
\fi

\if false
\begin{example} \label{ex.stars}
The stars $ {\mathcal S}_n$ on $n$ vertices are such that $V_{d_{{\mathcal S}_n}}(2) $ is dense in ${\mathbb Z}_{\geq 0}$ (see Theorem \ref{thm.denselist}).  The case of ${\mathcal S}_4$ (the Dynkin diagram $D_4$)
is discussed in Example \ref{ex.Dynkin4}, where computations indicate that the complement of ${\mathcal S}_4$ in ${\mathbb Z}_{\geq 0}$ is finite.
In the case of the star ${\mathcal S}_5$
(which is equal to the extended Dynkin diagram $\tilde{D_4}$),   computations leave open the possibility that the complement of $V_{d_{{\mathcal S}_5}}(2) $ in ${\mathbb Z}_{\geq 0}$ might not be finite.

The extended stars $ {\mathcal S}_n^+$, defined in the introduction,  are also such that $V_{d_{{\mathcal S}_n^+}}(2) $ is dense in ${\mathbb Z}_{\geq 0}$. 
 The graph ${\mathcal S}_4^+$ is equal to the Dynkin diagram $D_5$, and Example \ref{ex.Dynkin4}
suggests that the complement of $V_{d_{{\mathcal S}_4^+}}(2) $ in ${\mathbb Z}_{\geq 0}$ is finite. 
On the other hand, computations leave open the possibility that the complement of $V_{d_{{\mathcal S}_5^+}}(2) $ in ${\mathbb Z}_{\geq 0}$ might not be finite. Indeed, the complement in this case seems substantial, and its intersection with $[1,10^5]$ might contain as many as $6000$ elements. 
\end{example}

\begin{example}
Consider the complete bipartite graph $G={\mathcal K}(2,3)$, with $n=5$. 
Computations show that $d_G$ might fail to represents a set ${\mathcal C}$ of $693$ values in $[1,\dots, 10^5]$.

\if false 
Very surprisingly, $286$ of the potential $693$ missing values are primes (about $41\%$). 
For instance, this polynomial represents only the primes $3,47$, and $83$ among all primes up to $100$.
It fails to represents  $2 \cdot 3,2 \cdot 47$, and $2 \cdot 83$. 
In fact, the set ${\mathcal C}$ contains only $76$ even numbers, and $40 $ of them are of the form `twice a prime'. 
The set ${\mathcal C}$ contains only $37$ multiples of $3$ (with largest $4569$), and $19 $ of them are of the form `three times a prime'.
\fi  
\end{example}
\fi
\if false
We end this section by exhibiting arithmetical structures $(M,R)$ on stars and complete bipartite graphs. 
Once the order of the associated group $\Phi$ is computed, Proposition \ref{pro.arithmeticalstructure} can be used to predict known elements in the set $V_{d_G}(2)$. For instance, in the case of the bipartite graph $G={\mathcal K}(2,3)$,   Proposition \ref{pro.bipartite} implies that $V_{d_G}(2)$ contains all positive multiples of $4$, $18$, $25$, and $27$.
Proposition \ref{pro.bipartite} (b), with $r=3$ and $s=2$, describes a structure with
 a cyclic group $\Phi={\mathbb Z}/3{\mathbb Z}$ and this shows that $0 \in V_G(2)$.

\if false
In the entries of $V_{d_G}(2)$ up to $50$, we also find 
$3, 14, 33,$ and $47$, but $V_{d_G}(2)$ does  not contain all multiples of these integers.

Indeed, Proposition \ref{pro.bipartite} (a) implies, using the factorization $4=n-1=2 \cdot 2$, that $V_{d_G}(2)$ contains all positive multiples of $4=2^{n-3}$. 

The structure in Proposition \ref{pro.bipartite} (b) with $r=3$ and $s=2$,
\if false
$$\left( \begin{array}{ccccc}
  2 &- 1  &-1  &-1&  0\\
 -1 &3  &0 & 0  &-1\\
 -1 & 0&3 & 0  &-1\\
 -1 & 0 & 0 &3  &-1\\
 0  &-1  &-1  &-1& 2 
 \end{array}
 \right)
 \left( \begin{array}{c} 3 \\ 2\\ 2\\ 2\\ 3
 \end{array}
 \right)=0
 $$
 \fi 
has cyclic group $\Phi={\mathbb Z}/3{\mathbb Z}$ and shows that $0 \in V_G(2)$, and that $V_{d_G}(2)$ contains all positive multiples of $27$. The same structure when $r=2$ and $s=3$ is the Laplacian, with number of spanning trees $18$.

Proposition \ref{pro.bipartite} (c) produces a structure with $\Phi =(0)$ and whose vector $R$ has a coefficient equal to $5$, thus $V_{d_G}(2)$ contains all positive multiples of $25$.
\fi

  \begin{proposition} \label{pro.star}
Let $G$ be the star ${\mathcal S}_n$ on $n\geq 4$ vertices. 
If $n-1 $ is divisible by an integer $\ell$ with  $1<\ell<n-1$, then there exists on $G$ 
an arithmetical structure $(M,R)$
with associated group $\Phi=({\mathbb Z}/\ell{\mathbb Z})^{n-3}$, and $V_{d_{G}}(2)$ contains all positive multiples of $\ell^{n-3}$.
\end{proposition}
\begin{proof} Suppose that $v_1 $ is the vertex of degree $n-1$.  The matrix $M$ has diagonal 
$((n-1)/\ell, \ell, \dots, \ell)$. The transpose of the vector $R$ is $(\ell,1,\dots,1)$. 
Since $G$ is a tree, the order of $\Phi$ can be computed using \cite{Lor89}, Corollary 2.5, 
and is found to be $\ell^{n-3}$. The same corollary implies that $\Phi$ is killed by $\ell$. The last statement follows from Proposition \ref{pro.arithmeticalstructure}.
\end{proof}
 
 \begin{proposition} \label{pro.bipar}
 Any arithmetical structure $(M,R)$ on the star ${\mathcal S}_n$ on $n\geq 4$ vertices
 can be used to produce an arithmetical structure $(M',R')$ on the complete bipartite graph ${\mathcal K}(2,n-1)$ on $n+1$ vertices as follows. Let $v_0, \dots, v_{n-1}$ denote the vertices of the star, with $v_0$ the unique vertex of degree $n-1$.
 Denote by $(a_0,a_1,\dots, a_{n-1})$ the diagonal of $M$. 
 Let $^tR=(r_0,\dots, r_{n-1})$. Let  $w_0, \dots, w_{n-1}, w_n$ denote the vertices of the graph ${\mathcal K}(2,n-1)$, with $w_0$ and $w_n$ the only two vertices of degree $n-1$. Then $(M',R')$ is an arithmetical structure on ${\mathcal K}(2,n-1)$ when $M'$ has 
 diagonal $(a_0,2a_1,\dots, 2a_{n-1}, a_0)$ and 
$^tR'=(r_0,\dots, r_{n-1},r_0)$.

Let $\Phi$ and  $\Phi'$ denote the groups associated to $(M,R)$ and $(M',R')$, respectively.
Then
$|\Phi'|= 2^{n-2}a_0|\Phi|$.
 \end{proposition} 
\begin{proof} The proof is left to the reader. Note that the statement of Proposition \ref{pro.bipar} can be modified to apply to   ${\mathcal K}(q,n-1)$ for any $q \geq 2$.
\end{proof}
 
 \begin{proposition} \label{pro.bipartite}
Let $G$ be the complete  bipartite graph ${\mathcal K}(2,n-2)$ on $n \geq 3$ vertices. 
\begin{enumerate}[\rm (a)]
\item Factor $n-1=ab$ with $a,b\geq 1$. 
For each such factorization, there exists on $G$ 
an arithmetical structure 
with associated group $\Phi=({\mathbb Z}/a{\mathbb Z})^{n-3}$. 
It follows that when $a,b \geq 2$,  $V_{d_{G}}(2)$ contains all positive multiples of $a^{n-3}$ and $b^{n-3}$.
\item Factor $2(n-2)=rs$ with $r,s\geq 1$.
For each such factorization, there exists on $G$ an arithmetical structure 
with associated group 
$\Phi$ of order 
$|\Phi|=sr^{n-4}/2$ if $r$ is odd, and $|\Phi|=2sr^{n-4}$ if $r$ is even.
\item Let $n\geq 5$.
There exists on $G$ an arithmetical structure 
with associated group $\Phi$ of order 
$|\Phi|=n^{n-5}$ when $n$ is odd, and $|\Phi|=4n^{n-5}$ when $n$ is even.
\end{enumerate}
\end{proposition}
\begin{proof} For convenience, let $q:=n-2$. 
(a) For each factorization $q+1=ab$ in positive integers $a$ and $b$, we describe below an arithmetical structure $(M,R)$ on ${\mathcal K}(2,q)$.
The structure is given by the following matrix $M$ and associated vector $R$ such that $MR=0$:
$$
\left(
\begin{array}{ccccc}
b & -1& \dots & -1 & 0 \\
-1 & a& \dots & 0 & -1 \\
\vdots & 0 & \ddots & 0  & \vdots \\
-1 &0  & \dots & a  & -1 \\
0 & -1& \dots & -1 & qb \\
\end{array}
\right)
\left(
\begin{array}{c}
q \\
b\\
\vdots\\
b\\
1\\
\end{array}
\right)
= 0.
$$
To produce a row and column reduction of the matrix $M$ to compute the group $\Phi$, we first use the relation $MR=0$ to 
find a linear combination of the first $n-1$ columns
which, added to the last column, gives the $0$-vector.
Similarly,  a linear combination of the first $n-1$ rows
added to the last row gives the $0$-vector.
We are thus reduced to compute the row and column reduction of the $n-1 \times n-1$ matrix
$$
\left(
\begin{array}{cccc}
b & -1& \dots & -1  \\
-1 & a& \dots & 0  \\
\vdots & 0 & \ddots & 0   \\
-1 &0  & \dots & a   \\
\end{array}
\right) 
$$
Use the last row to eliminate the entries of the first column. Then use the first column to eliminate the last entry on the last row. Keep the top right submatrix, of size $q \times q$. Add all the columns to the last column. The top right coefficient is $-1+ab-(q-1)=1$.  Use the last column to eliminate the entries of the first row and keep the resulting bottom left 
matrix: it is a diagonal matrix ${\rm Diag}(a,\dots,a)$ of size $q-1$.

(b) For each factorization $2q =rs$ in positive integers $r$ and $s$, we describe below an arithmetical structure $(M,R)$ on ${\mathcal K}(2,q)$.
The structure is given by the following matrix $M$, and associated vector $V$ such that $MV=0$. When $r$ is odd, we take $R=V$, but when $r$ is even, we take $R=V/2$, so that  the $\gcd$ of the coefficients of $R$ is $1$:
$$
\left(
\begin{array}{ccccc}
s & -1& \dots & -1 & 0 \\
-1 & r& \dots & 0 & -1 \\
\vdots & 0 & \ddots & 0  & \vdots \\
-1 &0  & \dots & r  & -1 \\
0 & -1& \dots & -1 & s \\
\end{array}
\right)
\left(
\begin{array}{c}
r \\
2\\
\vdots\\
2\\
r\\
\end{array}
\right)
= 0.
$$ 
Note that when $r=2$ and $s=q$, this structure is the usual Laplacian of the graph.

(c) 
We describe below an arithmetical structure $(M,R)$ on ${\mathcal K}(2,q)$.
The structure is given by the following matrix $M$, and associated vector $V$ such that $MV=0$. When $n$ is odd, we take $R=V$, but when $n$ is even, we take $R=V/2$, so that  the $\gcd$ of the coefficients of $R$ is $1$. In the matrix $M$, the diagonal is $(n-2, n,\dots, n, 2,2,2)$ and the coefficient $n$ appears exactly $n-4$ times.
The transpose of the vector $V$ is 
 $(4, 2,\dots, 2, n,n,2n-4)$ and the coefficient $2$ appears exactly $n-4$ times.
$$
\left(
\begin{array}{ccccc}
n-2 & -1& \dots & -1 & 0 \\
-1 & n& \dots & 0 & -1 \\
\vdots & 0 & \ddots & 0  & \vdots \\
-1 &0  & \dots & 2  & -1 \\
0 & -1& \dots & -1 & 2 \\
\end{array}
\right)
\left(
\begin{array}{c}
4 \\
2\\
\vdots\\
n\\
2n-4\\
\end{array}
\right)
= 0.
$$ 
\end{proof}
\fi
 \end{section}

\begin{section}{Complete Graphs}
 
\begin{proposition} \label{complete}
Let $G={\mathcal K}_n$, the complete graph on $n \geq 2$ vertices. 
Then $1 \in V_G(2)$ if and only if the equation
\begin{equation} \label{Egypt}
\sum_{i=1}^n \frac{1}{y_i} + \frac{1}{y_1   \cdots   y_n} =1
\end{equation}
can be solved with positive integers $y_1,\dots,y_n \geq 3$.
\end{proposition}

\begin{proof}
Consider the square matrix 
$$M':=
\left( 
\begin{array}{cccc}
 1 & 0 & \dots & 0 \\
0 & && \\
\vdots && M_G(x_1,\dots, x_n)& \\
0 & && \\
\end{array}
\right)
$$
This matrix clearly has the same determinant as $M_G(x_1,\dots, x_n)$,
and is row and column equivalent to the matrix 
$$
M'':=\left( 
\begin{array}{cccc}
 1 & -1 & \dots & -1 \\
-1 &  x_1+1 &0& 0 \\
\vdots &0& \ddots & 0 \\
-1 & 0 & 0& x_n+1 \\
\end{array}
\right).
$$
Indeed, letting  
$$
T:=\left( 
\begin{array}{cccc}
 1 & 0 & \dots & 0 \\
-1 &  1 & & 0 \\
\vdots & & \ddots &   \\
-1 & 0 &  & 1 \\
\end{array}
\right),
$$
we find that $M''=TM'(^tT)$.
Note that if $H$ denote the star on $n+1$ vertices $w_0,\dots, w_n$, then the  matrix $M''$ is $M_H(1,x_1+1,\dots, x_n+1)$.
\if false
If  $^tZ:=(z_1,\dots, z_n)$ and  $^tZ':=(z_0,z_1,\dots, z_n)$, then $$^tZ'M'' Z' = (z_0-\sum_{i=1}^n z_i)^2 +(^t \! ZM Z).$$ \fi
It is clear that $M$ is positive definite if and only if $M'$, and hence $M''$, is
positive definite.   
Expanding the determinant of $M''$ using its first row, we obtain that
\begin{equation}
\label{det}
\det(M_G(x_1,\dots,x_n))=  \prod_{j=1}^n (x_j+1) -\sum_{i=1}^n \frac{\prod_{j=1}^n (x_j+1)}{x_i+1}.
\end{equation}

Assume now that $1 \in V_G(2)$. Then we can find integers $x_1,\dots, x_n \geq 2$ such that, setting 
$\ell:= \prod_{j=1}^n (x_j+1)$, we have
$$ 1=\ell -\sum_{i=1}^n \ell /(x_i+1).$$
Setting $y_i=x_i+1$, so that $\ell =\prod_{i=1}^n y_i$, we find  that $ 1/\ell=1 -\sum_{i=1}^n 1/y_i$ with $y_i \geq 3$ for all $i$, as desired.

Reciprocally, assume that there exist $y_1,\dots, y_n \geq 3$ such that
$ 1=\prod_{j=1}^n y_j -\sum_{i=1}^n (\prod_{j=1}^n y_j) /y_i$. Setting $x_i=y_i-1$, we obtain 
$x_1, \dots, x_n \geq 2$ such that $\det(M_G(x_1,\dots,x_n))=1$. 
To show that $1 \in V_G(2)$, it remains to show that $M_G(x_1,\dots,x_n)$ is positive definite. For this, it suffices, as noted above, to argue 
that $M''$ is  positive definite, and this can be obtained using Sylvester's Criterium on the positivity of the leading principal minors. This is clear since $\det(M'')=1>0$ in our case, and all the diagonal elements of $M''$
are positive.
\end{proof}

\begin{remark} In a solution $(y_1, \dots, y_n)$ 
to Equation \eqref{Egypt}, the integers $y_i$ are pairwise coprime. In particular, they are all distinct.
When $n$ is odd, any solution $(y_1, \dots, y_n)$ 
to Equation \eqref{Egypt}  must have at least one $y_i$ even.

When $n=13$, one finds in \cite{B-V}, page 8, a solution to  Equation \eqref{Egypt} in pairwise coprime integers
obtained by Girgensohn  with  $3=y_1<4<5<7<29<  \dots <  y_n $.  The integer $y_{13}$ has $172$ digits. This solution has 
exactly one even entry, $y_2=4$. Note that the same solution is given in \cite{B-J}, page 393, but in that article the given solution has typos.

Once we have a solution $(y_1, \dots, y_n)$ to Equation \eqref{Egypt}, it is possible to extend it
to a solution $(y_1, \dots, y_n, x)$ satisfying 
$$ \sum_{i=1}^n \frac{1}{y_i} + \frac{1}{x} + \frac{1}{y_1 \cdots y_n x}=1$$
by setting $x:= (y_1 \cdots y_n) +1$.  
\end{remark}

\begin{corollary} \label{cor.complete}
Let $G={\mathcal K}_n$ be the complete graph on $n\geq 2 $ vertices.
\begin{enumerate}[\rm (a)]
\item
If $n=13$, then  $1\in V_G(2)$. 
\item If $n\geq 14$, then $V_G(2) = {\mathbb Z}_{\geq 0}$.
\item If $n\leq 7$, then  $1\notin V_G(2)$.
\end{enumerate}
\end{corollary}
\begin{proof}
(a) In view of Proposition \ref{complete}, it suffice to show that 
the equation \eqref{Egypt} can be solved with $y_1,\dots, y_n \geq 3$. 
When $n=13$, we use the solution
provided in \cite{B-V}, page 8, and mentioned in the previous remark.

(b) The graph ${\mathcal K}_{n+1}$ is the cone on the graph ${\mathcal K}_n$.  
Therefore, Theorem \ref{thm.main1} (b) shows that if $1 \in V_{{\mathcal K}_n}(2)$, then $V_{{\mathcal K}_{n+1}}(2)={\mathbb Z}_{ \geq 0}
 $.
 
 (c) It is mentioned in \cite{B-V}, page 8, that all solutions $y_1 \leq \dots \leq y_n$ to Equation \eqref{Egypt} with $n \leq 9$ have $y_1=2$. We have not been able to retrieve the list of known solutions for $n=8,9$ to verify this claim. The list of solutions with $n \leq 7$ is provided in \cite{B-H}, page 50 and Appendix. 
 The claim (c) then follows from Proposition \ref{complete}.
\end{proof}

\begin{remark} Let $G={\mathcal K}_n$ be a complete graph with
$3 \leq n\leq 13$.
Corollary \ref{cor.density} shows that the complement
 of $V_{d_G}(2)$ in ${\mathbb Z}_{\geq 0}$ has density $0$. It is natural to wonder whether the complement is finite.
We discussed already the case $n=3$ in \ref{ex.C3}. Computations in the case   $n=4$ also seem to indicate that the complement in this case might be quite substantial.
\end{remark}

\end{section}
 
 \begin{acknowledgement}
The author gratefully acknowledge funding support from the Simons Collaboration Grant 245522. He also thanks the referee for a thoughtful report, and  Carlos Alfaro and Joshua Stucky
for bringing to his attention the references \cite{GRT} and \cite{CS}, respectively.
\end{acknowledgement}

\end{document}